\documentclass[11pt,reqno]{amsart}
\usepackage[top=1.3in,bottom=1.3in,left=1.3in,right=1.3in]{geometry}
\usepackage{color}

\usepackage{amssymb}
\usepackage{amsmath}
\usepackage{amsfonts}
\usepackage{geometry}
\usepackage{graphicx}
\usepackage{mathrsfs,amssymb}

\usepackage[utf8]{inputenc}
\usepackage[T1]{fontenc}
\usepackage{lmodern}

\usepackage{hyperref} 
\usepackage{cleveref}

\usepackage{enumerate}

\newtheorem{theorem}{Theorem}[section]
\newtheorem{proposition}[theorem]{Proposition}

\newtheorem{lemma}[theorem]{Lemma}
\newtheorem{corollary}[theorem]{Corollary}

\theoremstyle{definition}

\theoremstyle{remark}
\newtheorem{remark}[theorem]{Remark}

\numberwithin{equation}{section}

\newcommand{\R}{\mathbb{R}}
\newcommand{\N}{\mathbb{N}}
\newcommand{\C}{\mathbb{C}}

\newcommand{\la}{\lambda}

\newcommand{\mc}[1]{\mathcal{#1}}

\title[]{Existence and stability of shrinkers for the harmonic map heat flow in higher dimensions}

\author{Irfan Glogi\'c}
\address{Department of Mathematics, University of Vienna, Oskar-Morgenstern-Platz 1, 1090 Vienna, Austria}
\email{irfan.glogic@univie.ac.at}
\author{Sarah Kistner}
\address{Universit\"at Innsbruck, Institut f\"ur Mathematik,\\ Technikerstraße 13, 6020 Innsbruck, Austria}
\email{Sarah.kistner@uibk.ac.at}
\author{Birgit Sch\"orkhuber}
\address{Universit\"at Innsbruck, Institut f\"ur Mathematik,\\ Technikerstraße 13, 6020 Innsbruck, Austria}
\email{Birgit.Schoerkhuber@uibk.ac.at}
\thanks{Irfan Glogi\'c is supported by the Austrian Science Fund FWF, Projects P 30076 and P 34378.}

\usepackage{kantlipsum}

\begin{document}
	\maketitle
\begin{abstract}
We study singularity formation for the heat flow of harmonic maps from $\R^d$. For each $d  \geq 4$, we construct a compact, $d$-dimensional, rotationally symmetric target manifold that allows for the existence of a corotational self-similar shrinking solution (shortly \emph{shrinker}) that represents a stable blowup mechanism for the corresponding Cauchy problem. 
\end{abstract}

\section{Introduction}

Given Riemannian manifolds $(M,\tilde{h})$ and $(N,h)$, a smooth map $U:M \rightarrow N$ is called \textit{harmonic} if it is a critical point (under compactly supported variations) of the energy functional 
\begin{align}\label{Energy_S}
S(U) := \frac{1}{2} \int_M |dU|^2 d\text{vol}_{\tilde{h}},
\end{align} 
where the differential $dU$ of the map $U$ is viewed as a section of the vector bundle $(T^*M \otimes U^*TN, h \otimes U^*\tilde{h})$. By means of local coordinates, this functional adopts a more explicit form\footnote{Here, and throughout the paper, we use the Einstein summation convention.
}
\begin{align} \label{1}
S(U) = \frac{1}{2} \int_M \tilde{h}^{ij}(x) h_{ab}(U(x)) \partial_i U^{a}(x) \partial_j U^{b}(x) \sqrt{|\tilde{h}(x)|} dx.
\end{align}
Furthermore, the Euler-Lagrange system corresponding to \eqref{1} is given by
\begin{align*}
\Delta_{\tilde{h}} U^{a} + \tilde{h}^{ij} \Gamma^{a}_{bc}(U) \partial_i U^{b} \partial_j U^{c} = 0,
\end{align*}
where $\Delta_{\tilde{h}}$ denotes the Laplace-Beltrami operator on $M$, and $\Gamma_{bc}^{a}$ are the Christoffel symbols that correspond to metric $h$.

The study of harmonic maps is remarkably rich and it is impossible to survey all the relevant literature. We therefore point the reader to the following works and references therein \cite{EelLem78,EelLem88,HelWoo08,LinWan08}.  A particularly prominent approach to the question of the existence of harmonic maps was put forward in 1964 by Eells and Sampson \cite{EelSam64}. Namely, they considered the negative $L^2$-gradient flow of \eqref{1} 
\begin{align} \label{HMHF}
\partial_t U^{a} - \Delta_{\tilde{h}} U^{a} - \tilde{h}^{ij} \Gamma^{a}_{bc}(U) \partial_i U^{b} \partial_j U^{c}   = 0, \quad t >0, 
\end{align}
for arbitrary initial map $U(0,\cdot)  =U_0$. 
Due to the parabolic character of \eqref{HMHF}, the evolution is expected to converge to a static solution, which represents a harmonic map. This approach turned out to be very efficient under certain curvature assumptions on the target manifold. In particular, it is shown in \cite{EelSam64} that if $M$ is closed and $N$ is compact with non-positive sectional curvature, then \eqref{HMHF} admits for arbitrary smooth initial data a global solution which converges to a harmonic map as $t \to \infty$. 
If the curvature of the target $N$ has unrestricted sign, global existence is known for ``small'' data \cite{Jos81}. In general, however, the flow \eqref{HMHF} might form singularities in finite time, as demonstrated first in \cite{ChaDinYe92, CorGhi89} for maps into the sphere.

In order to continue solutions  past blowup,
in a possibly weaker sense, one requires a detailed understanding of the nature of singularities.  Moreover, due to the local nature of the blowup, it suffices to consider maps from the tangent space of the domain manifold, namely $\R^d$; see, e.g., \cite{Str88,GraHam96,Fan99} for a discussion. 
In case $d \geq 3$, a characterization due to Struwe \cite{Str88} provides the resolution of blowup solutions (along a sequence of times) into limiting profiles, which are either given by non-constant harmonic maps from $ \mathbb{S}^{d-1}$ into $N$, or by a profile coming from a self-similar solution. For $N = \mathbb{S}^{d}$, explicit examples of blowup of each type have been constructed \cite{Fan99, Gas02, Bie15, BieBiz11, BieSek19, BieSek20},  and some even lead to stable blowup dynamics \cite{BieDon18, BieDonSch17, GhoIbrNgu19}.
However, for general target manifolds, the question of existence and stability of blowup, and in particular of self-similar singularities, is largely open. 

In the following, we restrict our attention to $d \geq 3$ and to rotationally symmetric target manifolds. More precisely, we study \eqref{HMHF}  for maps from $\R^d$ into $N$, where $(N,h)$ is a  $d$-dimensional, complete, rotationally symmetric, warped product Riemannian manifold
\begin{align} \label{Nd}
(N,h) = (0,a^*) \times_g \mathbb{S}^{d-1} \quad \text{for some} \quad a^* \in (0, \infty],
\end{align}
where  $g$ is a suitable warping function; see, e.g., \cite{Tac85,Che17}.  In polar coordinates $(u, \Omega) \in (0, a^*) \times \mathbb{S}^{d-1}$ on $N$, the metric $h$ is given by
\begin{align} \label{metrich}
h = du^2 + g(u)^2 d \Omega^2,
\end{align}
where $d \Omega^2$ is the standard round metric on the sphere $\mathbb{S}^{d-1} \hookrightarrow \R^{d}$. Furthermore, the warping function $g$ must satisfy the following conditions
\begin{equation}\label{Eq:g_properties}
	g\in C^{\infty}(\R,\R),\quad g \text{ is odd,} \quad g'(0)=1.
\end{equation}
The point that corresponds to the limiting value $u=0$ is referred to as \emph{the vertex} of $N$. In case of compactness of $N$, i.e., if $a^* < \infty$ is the smallest positive zero of $g$, then, in addition to \eqref{Eq:g_properties}, we require that $g'(a^*)=-1$ and that $g$ extends $2a^*$-periodically beyond the interval $(-a^*,a^*]$.
Rotational symmetry of the target allows us to restrict the attention to \emph{corotational} maps from $\R^d$ into $N$. Namely, by introducing polar coordinates $(r,\omega)$, $r > 0$, $\omega \in \mathbb S^{d-1}$, on the domain, we restrict ourselves to maps of the form  
\begin{align}\label{Corot}
(u(t,r,\omega),\Omega(t,r,\omega)) = (u(t,r), \omega).
\end{align}
With this ansatz, the system 
\eqref{HMHF} reduces to a single PDE for $u$. In particular, by setting  $r v(t,r) := u(t,r)$, the initial value problem for \eqref{HMHF}  is equivalent to 
\begin{equation}\label{Eq:CorHMHF}
	\begin{cases}
		~~\displaystyle{\partial_{t}v-\partial_r^2 v - \frac{d+1}{r} \partial_r v  + \frac{d-1}{r^3}\big(g(rv)g'(rv)-rv\big)=0, \quad t > 0},\\[1mm]
		~~v(0,\cdot)=v_0.
	\end{cases}
\end{equation}

From the analytic point of view it is convenient to consider on $N$ so-called \emph{normal coordinates} $U = (U^1,\dots,U^d)$, where
\begin{equation}\label{Def:NormCoord}
	U^j:=u\, \Omega^j, \quad \text{for} \quad j=1,\dots,d;
\end{equation}
see, e.g., \cite{ShaTah94}. In this way, $N$ (including its vertex) can be identified with the ball $B^d_{a^*}(0) \subset \R^d$. Then, in normal coordinates, a corotational solution  $U(t,\cdot): \R^d \to \R^d$ to \eqref{HMHF} can be written as
\begin{align}\label{Def:CorAnsatz}
U(t,x) = u(t,|x|) \frac{x}{|x|} = x  v(t,|x|),
\end{align}
where $v$ solves \eqref{Eq:CorHMHF}. We take this point of view in the following.

\subsection{Main results}

\subsubsection{Existence of shrinkers}
One way of exhibiting blowup is via self-similar solutions.
Since the system \eqref{HMHF} is invariant under the scaling
\[ U(t,x) \mapsto U_{\lambda}(t,x) := U(t/\lambda^2, x/\lambda), \quad \lambda > 0, \] 
it is plausible to look for non-trivial solutions of the following form
\begin{align}\label{Def:Corot_Shrinker}
U_T(t,x) = \Phi\left (\frac{x}{\sqrt{T-t}} \right ), \quad T>0, 
\end{align}
for a profile $\Phi: \R^d \to N$. This type of self-similar (shrinking) solutions are also referred to as \emph{shrinkers} or \emph{homothetically shrinking solitons}.
In the corotational case, $U_T$ corresponds to 
\begin{align} \label{5}
	v_T(t,r) = \frac{1}{\sqrt{T-t}} \phi \left( \frac{r}{\sqrt{T-t}} \right),
\end{align}
where in normal coordinates the similarity profiles are related by $\Phi(x) = x \phi(|x|)$. 

In this paper, we study the existence and stability of solutions \eqref{Def:Corot_Shrinker}. The minimal requirements we impose on the profile $\Phi: \R^d \rightarrow \R^d$ are those of being smooth and bounded. 
By inserting the ansatz \eqref{5} into \eqref{Eq:CorHMHF} we obtain an  ODE for the profile $\phi = \phi(\rho)$,
\begin{align} \label{6}
\phi^{\prime \prime} + \left( \frac{d+1}{\rho} - \frac{\rho}{2} \right) \phi^{\prime} - \frac{1}{2} \phi - \frac{d-1}{\rho^3} \Big(g\big(\rho \phi\big)g^{\prime}\big(\rho \phi\big)- \rho \phi\Big) = 0.
\end{align}
The existence of shrinkers depends, of course, on the geometry of the target $N$, which is in turn completely encoded in the warping function $g$. From \cite{EelSam64,Har67}, we know that the assumption that $N$ is non-positively curved precludes the formation of singularities.
The property of having non-positive sectional curvature for manifolds of type \eqref{Nd}-\eqref{metrich} is simply characterized by the condition that
\begin{equation*}
	g''(u) \geq 0 \quad \text{for all} \quad u \in [0,\infty);
\end{equation*}
see, e.g., \cite{DonGlo19}. One natural question would be as to whether (and how)  this class of manifolds can be enlarged so as to still preclude the existence of blowup in general, and shrinkers in particular. An obvious candidate is the family of manifolds that are called ``geodesically convex'' by Shatah and Tahvildar-Zadeh in \cite{ShaTah94}, and are characterized by
\begin{equation}\label{Eq:Geod_convex}
	g'(u) > 0 \quad \text{for all} \quad u \in [0,\infty).
\end{equation}
We remark that the existence of self-similar blowup for this class of manifolds in the setting of wave maps, the hyperbolic analogue of the harmonic map heat flow, has already been studied in  \cite{ShaTah94}. Using finite-speed of propagation, the authors in particular prove that for every $d \geq 4$ there exists a rotationally symmetric target satisfying \eqref{Eq:Geod_convex} that allows for a corotational self-similar blowup solution.
As it turns out, for the harmonic map heat flow this is not the case. We state this in the form of a proposition, and provide a short proof in Section \ref{Sec:Proof_Prop_1}.

\begin{proposition} \label{NdPropo}
Let $d \geq 2$ and suppose $(N, h)$ is a $d$-dimensional warped product Riemannian manifold given by \eqref{Nd}-\eqref{metrich}. If the warping function $g$ satisfies \eqref{Eq:Geod_convex}, then the heat flow of harmonic maps from $\R^d$ into $N$ does not admit a non-trivial, smooth and bounded corotational self-similar shrinking solution. Boundedness here stands for the image of the shrinker being contained inside a bounded neighborhood of the vertex of $N$.
\end{proposition}
Note that Proposition \ref{NdPropo} applies to a large class of positively curved manifolds, namely those that have an increasing but strictly concave warping function $g$. Furthermore, it follows that for a manifold \eqref{Nd}-\eqref{metrich} to admit a corotational shrinker, a necessary condition is that it contains an ``equator'', i.e., there exists $u^*>0$ such that $g'(u^*)=0$. One prominent manifold of that type is, of course, the sphere. What is more, it is known that the heat flow of harmonic maps from $\R^d$ into the $d$-sphere  admits for all $3 \leq d \leq 6$ a corotational shrinker; see, e.g., \cite{Fan99,Gas02}. However, this seizes to be true for $d \geq 7$; see \cite{BizWas15}. To the best of knowledge of the authors, there are no known examples of target manifolds that admit a corotational shrinker for $d \geq 7$. The purpose of the first of the two main result of this paper is to fill this gap.
\begin{theorem} \label{geotheo}
Let $0 < \gamma < \sqrt{2} -1$. For every $d \geq 4$ there exists a compact, $d$-dimensional, rotationally symmetric Riemannian manifold $(N, h)$ given by \eqref{Nd}-\eqref{metrich}, such that the following holds: The heat flow of harmonic maps from $\R^d$ into $N$ admits a non-trivial corotational shrinker 
\begin{align*}
U_T(t,x) = \Phi \left( \frac{x}{\sqrt{T-t}} \right),
\end{align*}
where the profile $\Phi :\R^d \rightarrow N$ is smooth and furthermore, 
\begin{align} \label{Linfty}
d(\Phi(x),p) \leq r_g + \gamma
\end{align}
for all $x \in \R^d$, where $p$ is the vertex of $N$ and $r_g$ is the least positive critical point of the warping function $g$.
\end{theorem}

We defer the proof to Section \ref{Sec:Main_proof_1}, although we will refer to it prior to that in the statement of the Theorem \ref{maintheorem} below.

\begin{remark}
		Note that Proposition \ref{NdPropo} implies that the image of any corotational shrinker strictly contains the closed geodesic ball centered at the vertex of the target and character- ized by $g'(u)>0$. 
Theorem \ref{geotheo}, on the other hand, shows that this just marginally holds in general. Namely, there are target manifolds
that allow for a shrinker such that the above described geodesic ball just barely fails to contain the image of the shrinker. In particular, \eqref{Linfty} implies that in normal coordinates
\[ \| \Phi \|_{L^{\infty}(\R^d)} \leq  r_g + \gamma. \]	
\end{remark}

\subsubsection{Stability of shrinkers}

Given a blowup solution, the natural question is whether it is stable under small perturbations. We show that the shrinker $U_T$ constructed in the proof of Theorem \ref{geotheo} is nonlinearly asymptotically stable under small corotational perturbations of the initial datum. More precisely, we prove that there is an open set (in a suitable topology) of corotational initial data around $U_1(0,\cdot)=\Phi$, for which the Cauchy evolution of \eqref{HMHF} forms a singularity in finite time $T>0$ by converging to $U_T$, that is, to the profile $\Phi$, after self-similar rescaling. For reasons that are explained below, we fix a value for $\gamma$ and restrict ourselves to a certain range of dimensions.
\begin{theorem} \label{maintheorem}
Let $\gamma = \frac{1}{4}$, $d \in \{4, 5, 6, 7 \}$, and $(N,h)$ and $U_T$ be the Riemannian manifold and the corresponding shrinker constructed in the proof of Theorem \ref{geotheo}. Assume that $s,k>0$ satisfy 
\begin{align} \label{8}
\frac{d}{2} < s \leq \frac{d}{2} + \frac{1}{2d}, \quad k > d+2, \quad k \in \N.
\end{align}
Then, there exists $\varepsilon >0$ such that for any corotational initial datum of the form 
\begin{align*}
U_0 = \Phi + \eta_0,
\end{align*}
where $\eta_0: \R^d \to \R^d$ is a Schwartz function satisfying 
\begin{equation*}
	\| \eta_0  \|_{\dot{H}^{s} \cap \dot{H}^{k}(\R^d)} \leq \varepsilon,
\end{equation*}
there exists $T>0$ and a classical solution $U \in C^{\infty}([0,T) \times \R^d)$ to \eqref{HMHF}, whose gradient blows up at the origin as $t \rightarrow T^-$. Furthermore, $U$ can be decomposed in the following way
\begin{align}\label{Main:Sol_Decomp}
U(t,x) = \Phi \left(\frac{x}{\sqrt{T-t}} \right) + \eta \left(t, \frac{x}{\sqrt{T-t}} \right),
\end{align}
where for any  $r \in [s,k]$ we have that
\begin{equation}\label{Eq:varphi_zero}
	\Vert \eta(t, \cdot) \Vert_{\dot{H}^r(\R^d)}\rightarrow 0
\end{equation}
in the limit $t \rightarrow T^-$.
\end{theorem}

We follow with several remarks.

\begin{remark}
From \eqref{Eq:varphi_zero} it follows that
	\begin{align*}
		\| U(t, \sqrt{T-t} \cdot) - \Phi \|_{\dot{H}^r(\R^d)} \rightarrow 0, \quad \text{ for } t \rightarrow T^-.
	\end{align*}
In other words, the solution $U$, which starts off as a small deviation from $\Phi$, converges upon self-similar rescaling  back to $\Phi$ in a suitably chosen  topology. This corresponds to what is conventionally meant by stability of self-similar solutions. Based on the restrictions on the choice of the lower Sobolev exponent $s$ in \eqref{8}, we bairly fail to control the $L^{\infty}(\R^d)$ norm of the perturbation $\eta$, which means, that it could in principle, grow uncontrollably pointwise and thereby wrap around the target manifold multiple times. However, by Sobolev embedding, we infer that $\|\nabla \eta(t,\cdot)\|_{L^{\infty}(\R^n)} \to 0$, which implies that in the limit $t \to T^{-}$ we have
	 	\begin{align*}
		\sqrt{T-t} \|\nabla U(t, \sqrt{T-t} \cdot) - \nabla \Phi \|_{L^{\infty}(\R^d)} \rightarrow 0.
	\end{align*}
\end{remark}

\begin{remark}
The proof of Theorem \ref{maintheorem} hinges on the spectral properties of the operator representing the linearization around the shrinker. Here, obtaining results that are uniform in $d$ appears to be difficult. Hence, in order to provide a fully rigorous proof, we restrict ourselves to lower space dimensions and a fixed value of $\gamma$. However,  an analogous result to Theorem \ref{maintheorem} can be obtained in any given dimension $d \geq 4$ and $\gamma \in (0, \gamma^*)$ for some $\gamma^* = \gamma^*(d) < \sqrt{2}-1$.\\
Additionally, the restrictions on $s,k$ in \eqref{8} guarantee exponential decay of the linearized flow on a suitable subspace as well as the local Lipschitz-continuity of the nonlinearity produced by the warping function $g$; see Sections \ref{Sec:Semigroup_X} and \ref{Sec:Nonlin}.
\end{remark}

\begin{remark}
The proof Theorem \ref{maintheorem} is similar in spirit to the approach developed in \cite{YangMills} by the first and the third author for the Yang-Mills heat flow in $d \geq 5$; see also \cite{BieDonSch17} for the three-dimensional harmonic map heat flow into the sphere by Biernat, Donninger and the third author. However, in contrast to \cite{BieDonSch17,YangMills}, we cannot work entirely in Sobolev spaces of integer order. In particular, in order to control the evolution in an intersection Sobolev space $\dot H^{s} \cap \dot H^{k}$, we necessarily have $s \notin \N$; see Section \ref{Sec:Outline} below for more details.  In turn, by generalizing the approach to this situation, we provide a framework for studying stability of singularity formation via shrinkers for the harmonic map heat flow in the corotational case for arbitrary dimension $d \geq 3$. In particular, modulo solving a spectral problem, our methods can be applied to extend the result of \cite{BieDonSch17} to $d \in \{4,5,6\}$. 
\end{remark}

\subsection{Related results}

The analysis of the heat flow of harmonic maps is a vast subject, and it is impossible to review all of the relevant works here. Therefore, we concentrate on the results about singularity formation that are related to our work, while for the general background we point the reader to, e.g., \cite{Str92,LinWan08}. We emphasize that we survey primarily the papers that concern maps from $\R^d$ into rotationally symmetric manifolds.

First we note that for corotational maps the energy functional  \eqref{Energy_S} reduces (up to a constant) to 
\[ \mc E(u) =  \frac{1}{2} \int_0^{\infty} r^{d-1} \left ( (\partial_r u)^2 + \frac{d-1}{r^2} g(u)^2 \right ) dr. \]
With the natural scaling $u_{\lambda}(t,r) = u(t/\lambda^2,r/\lambda)$ we obtain $\mc E(u_{\lambda})(t) = \lambda^{d-2} \mc E(u)(t/\lambda^2)$, which shows that the problem is energy critical in $d=2$ and supercritical in $d \geq 3$.

In the energy critical case, Rapha\"el and Schweyer proved in \cite{RapSch13} that there is a large class of rotationally symmetric targets (including the 2-sphere) that allow for corotational blowup solutions. What is more, they obtained stability of the underlying blowup mecha-nism; see also \cite{RapSch14} for a refined description of possible blowup regimes that include unstable ones as well. As is typical for the problems of critical type, the blowup in the aforementioned papers is non-self-similar, and takes place via rescaling of a harmonic map. For a classical result on blowup for maps from bounded 2-dimensional domains, see 
\cite{ChaDinYe92}, and for more recent results, see
\cite{DavDelPes19,DavPinWei20}.

In the energy supercritical case, $d \geq 3$, the first construction of blowup was provided by Coron and Ghidaglia \cite{CorGhi89}, for maps into $\mathbb{S}^d$. Subsequently, Fan \cite{Fan99} showed that within the class of corotational maps into $\mathbb{S}^d$ there are infinitely many self-similar blowup solutions, but only for the restricted range $3 \leq d \leq 6$. For $d \geq 7$, later on Bizo\'n and Wasserman \cite{BizWas15} proved that shrinkers are in fact absent.
Gastel, on the other hand, showed in \cite{Gas02} that if one allows for non-corotational maps, namely those of higher equivariance classes that take values in spheres of dimensions  strictly higher than the one of the domain, then there are self-similar solutions for any $d\geq3$.  In terms of stability of self-similar blowup, the only known works to the authors are \cite{BieDon18}, where Biernat and Donninger construct a spectrally stable self-similar solution in $d=3$,  and \cite{BieDonSch17}, where they together with the third author of this paper prove that the constructed blowup profile is nonlinearly stable.
Existence and stability for non-self-similar blowup was considered in several works, in particular by Biernat \cite{Bie15}, Ghoul, Ibrahim and Nguyen \cite{GhoIbrNgu19},  and Biernat and Seki \cite{BieSek20} for maps into the sphere. For results discussing the question of continuation beyond blowup we refer the reader to \cite{BieBiz11,GerGhoMiu17}.

\subsection{Proof of Proposition \ref{NdPropo}}\label{Sec:Proof_Prop_1}

	We argue by contradiction. Assume there is a non-trivial corotational map $\Phi : \R^d  \rightarrow N$ that is smooth, bounded, and such that \eqref{Def:Corot_Shrinker} solves \eqref{HMHF} for $t < T$. This, in particular, means that in the normal coordinates on $N$, we have that $\Phi(x)=x \phi(|x|)$, where $\phi:[0,\infty) \rightarrow \R$ is smooth and satisfies \eqref{6}. Additionally, boundedness of $\Phi$ implies boundedness of $\rho \mapsto \vartheta(\rho):=\rho \phi(\rho)$ on $[0,\infty)$.
	
	Now, without loss of generality, we can assume that $\phi(0)>0$, as otherwise by reflection symmetry we can consider $-\phi$. Note that \eqref{Eq:Geod_convex}  and \eqref{Eq:g_properties} imply that 
	\begin{equation}\label{Def:f=gg'}
		f(u):=g(u)g'(u)>0 \quad \text{for} \quad u > 0.
	\end{equation}
	Also note that \eqref{6} implies
	\begin{equation}\label{Eq:Lyapunov}
		\frac{d}{d\rho}\left(\rho^{d-1}e^{-\frac{\rho^2}{4}}\vartheta'(\rho)\right)=(d-1) \rho^{d-3}e^{-\frac{\rho^2}{4}}f\big(\vartheta(\rho)\big).
	\end{equation}
	Since $\phi(0)>0$, we have that $\vartheta(\rho)>0$ for small positive values of $\rho$. Based on this, we conclude that $\vartheta$ must, in fact, be positive on the whole interval $(0,\infty)$. Indeed,  there would otherwise be the smallest $\rho^*>0$ such that $\vartheta(\rho^*)=0$, and integrating  \eqref{Eq:Lyapunov} on $(0,\rho^*)$ would yield a non-positive number on the left and positive on the right. Now, as $\vartheta$ is globally positive, from \eqref{Eq:Lyapunov} we have that $\rho \mapsto \rho^{d-1}e^{-\rho^2/4}\vartheta'(\rho)$ is increasing on $(0,\infty)$.
	Consequently, since there is a small enough $\rho_0>0$ such that $C:=\rho_0^{d-1}e^{-\frac{\rho_0^2}{4}}\vartheta'(\rho_0) >0$, we have that
	\begin{equation*}
		\vartheta'(\rho) \geq C \rho^{1-d}e^{\frac{\rho^2}{4}}
	\end{equation*}
	for $\rho \geq \rho_0$. From here, we get that $\vartheta(\rho) \rightarrow +\infty$ as $\rho \rightarrow +\infty$, which is in contradiction with boundedness of $\Phi$.

\subsection{Proof of Theorem \ref{geotheo}}\label{Sec:Main_proof_1}
	To construct the manifold $(N,h)$, we only need to specify the warping function $g$. To begin, we fix $0 < \gamma < \sqrt{2}-1$. For $d \geq 4$ we define $g$ for small values of the argument in the following way
	\begin{align} \label{7}
		g(u) := u \sqrt{1-\alpha u^2 + \beta u^4},
	\end{align}
	where
	\begin{align} \label{alphabeta}
		\alpha = \frac{3}{2(d-1)(1+\gamma)^4} + \frac{1}{2}, \quad \text{and} \quad \beta = \frac{1}{(d-1)(1+\gamma)^4}.
	\end{align}
	For this choice of $g$, the corresponding shrinker equation 
	\begin{align*}
		\phi^{\prime \prime}(\rho) + \left( \frac{d+1}{\rho} - \frac{\rho}{2} \right) \phi^{\prime}(\rho) - \frac{1}{2} \phi(\rho) + (d-1)\big(2 \alpha \phi(\rho)^3 - 3 \beta \rho^2 \phi(\rho)^5 \big) =0
	\end{align*} 
	admits an explicit solution
	\begin{align} \label{Shrinker}
		\phi(\rho) = \frac{a}{\sqrt{\rho^2 + b}},
	\end{align}
	where
	\begin{align} \label{ab}
		a = 1+ \gamma \quad \text{and} \quad b = \frac{2 \gamma (2+ \gamma) \big((d-1)(1+ \gamma)^2 -3\big)}{(1+ \gamma)^2}.
	\end{align}
	It is straightforward to conclude that the least positive zero of $g^{\prime}$ is given by $r_g =1$. Furthermore, we have that  $\phi$ is smooth and 
	\begin{equation}\label{Eq:gamma_sup}
		\sup_{\rho \geq 0} |\rho \phi(\rho)| = 1+\gamma.
	\end{equation} 
	We now also prove that $g(u) >0$ for $u \in (0, 1+\gamma]$. If $\alpha^2 - 4 \beta < 0$ then $g(u) >0$ for all $u >0$ and in particular for $u \in (0, 1+\gamma]$. Otherwise, the least positive zero of $g$ is given by
	\begin{align*}
		u^* = \sqrt{\frac{\alpha- \sqrt{\alpha^2-4 \beta}}{2 \beta}},
	\end{align*}
	and it is therefore enough to prove that $1+ \gamma < u^*$. By simple algebraic manipulation we see that this is equivalent to
	\begin{align*}
		(d-1)( 1 + \gamma)^4 - 2d(1+ \gamma)^2 +3 < 0,
	\end{align*}
	which holds when $d \geq 4$ and $\gamma < \sqrt{2} -1$. 
	Now, we simply extend $g$ beyond $u=1+\gamma$ in a way that yields a compact manifold. In conclusion, we constructed a warping function $g$ that defines a compact manifold $(N,h)$, and such that the shrinker ODE \eqref{6} admits an explicit solution \eqref{Shrinker} (note that this is ensured by \eqref{Eq:gamma_sup}). Furthermore, from  \eqref{Eq:gamma_sup} we deduce \eqref{Linfty}.

\subsection{Outline of the proof of Theorem \ref{maintheorem}}\label{Sec:Outline}

Since the heat flow of corotational perturba-- tions of $U_T$ is governed by the $(d+2)$-dimensional radial heat equation \eqref{Eq:CorHMHF}, we let $n:=d+2$ and consider the Cauchy problem for the corresponding $n$-dimensional semilinear heat equation in $w(t,x):=v(t,|x|)$
\begin{equation*}
	\begin{cases}
	~\displaystyle{\partial_t w - \Delta w = \frac{n-3}{|x|^3}\Big(|x|w-g\big(|x|w\big)g'\big(|x|w\big)\Big),}  \quad t > 0,\\[2mm]
	~w(0,\cdot)=w_0(|\cdot|),
	\end{cases}
\end{equation*} 
which admits an explicit self-similar solution
\begin{equation*}
	w_T(t,x)= \tfrac{1}{\sqrt{T-t}}\phi \left( \tfrac{|x|}{\sqrt{T-t}} \right),
\end{equation*}
with $\phi$ given by \eqref{Shrinker}-\eqref{ab}. Consequently, the majority of our work consists of proving stability of $w_T$ under small radial perturbations, which we then translate into a result about $U_T$ by using the equivalence of Sobolev norms of corotational maps and those of their radial profiles. The starting point are the similarity variables
\begin{equation*}
	\tau := \ln \left( \tfrac{T}{T-t} \right) \quad \text{and} \quad y := \tfrac{x}{\sqrt{T-t}}.
\end{equation*}
By this, and the scaling of the dependent variable $\psi( \tau, y) = \sqrt{T-t} \, w(t,x)$, self-similar shrinking solutions become static, i.e., $\tau$-independent, and thereby the problem of finite time stability of blowup becomes the one of the asymptotic stability of a steady state profile. In particular, assuming $\psi(\tau,\cdot) = \phi(|\cdot|) + \varphi(\tau,\cdot)$, we obtain an equation for the perturbation $\varphi$,
\begin{equation}\label{Eq:Evol_eq_outline} 
	\begin{cases}
		~\partial_\tau \varphi (\tau, \cdot) = L \varphi(\tau, \cdot) + \mathcal{N}(\varphi(\tau, \cdot)), \quad \tau > 0,\\[1mm]
		~\varphi(0,\cdot) = \sqrt{T}  w_0( \sqrt{T} |\cdot|) - \phi(|\cdot|),
	\end{cases}
\end{equation}
with 
\begin{equation*}
	Lf(x) = \Delta f(x) - \tfrac{1}{2}x\cdot \nabla f(x) - \tfrac{1}{2}f(x)  + V(x)f(x),
\end{equation*}
where the potential $V$ comes from linearizing around $\phi(|\cdot|)$, and $\mc N$ is the nonlinear remainder. First, one notices that $L$ can be realized as a self-adjoint operator on a Hilbert space $\mc H$, which corresponds to a weighted $L^2$-space of radial functions, see \eqref{L2sigma}.
However, since the weight function is exponentially decaying, it is impossible to control the nonlinearity in such a setting. Instead, we study the evolution in the intersection radial Sobolev space
\begin{equation*}
	X_{s}^{k}(\R^n) = \dot{H}^s_{r}(\R^n) \cap \dot{H}^k_{r}(\R^n),
\end{equation*}
with $0 < s-s_c \ll 1$, where $s_c=n/2-1$ is the critical Sobolev exponent, and $k \in \N$,  $k \gg 1$. The conditions on $s,k$ are dictated by the following requirements. First, the choice $s > s_c = \frac{n}{2}-1$ ensures exponential decay of the linearized evolution on a suitable subspace.  Furthermore, to obtain a reasonable space of distributions, we need $s < \frac{n}{2}$, and thus, $s$ has to be non-integer. In addition, if $k > \frac{n}{2}$, then  $	X_{s}^{k}$ embeds continuously into $L^{\infty}(\R^n)$. The stronger assumptions implied by \eqref{8} are due to an application of a generalized Schauder estimate (\cite{WMglobal}, Proposition A.1), which allows us to prove the local Lipschitz property of the nonlinearity. 

In $X_{s}^{k}$, the self-adjoint structure  of the linearized problem is lost and by that the standard self-adjoint spectral and semigroup techniques become inaccessible. However, exploiting the fact that $X_{s}^{k}$ embeds continuously into $\mc H$, see Lemma \ref{embeddinglemma}, allows us to transfer results on the spectrum of $L$ in $\mc H$ into growth bounds for the semigroup generated by $L$ in $X_{s}^{k}$. In this step, the fractional nature of $s$ poses severe difficulties compared to previous works, see in particular Lemma \ref{EstimateVf} and Proposition \ref{Semigroupdecaypropo}.

To analyze the spectral properties of $L$ in $\mc H$, we study the corresponding Schr\"odinger operator and use a supersymmetric approach to show that the spectrum of $L$, which consists only of real isolated eigenvalues, is confined to the left half plane, except for $\lambda = 1$. However, this eigenvalue is an artifact of the time translation symmetry, and therefore not a genuine instability. Consequently, $L$ generates a semigroup on $\mc H$, which decays exponentially on the stable subspace orthogonal to the unstable mode $G$. Since $G \in X_{s}^{k}$, the orthogonal projection onto $G$ in $\mc H$  gives rise a (non-orthogonal) projection in $X_{s}^{k}$ and by this, we are able to prove exponential decay of the linearized evolution in $X_{s}^{k}$ on an invariant subspace. 

Following this, we employ a fixed point argument, where by a Lyapunov-Perron type argument we show that for every initial datum $w_0$ that is close enough to $\phi(|\cdot|)$ there exists a choice of time $T$ near 1 that yields a global and exponentially decaying solution to \eqref{Eq:Evol_eq_outline}. By using regularity arguments, we show that smooth and rapidly decaying initial datum $\varphi(0,\cdot)$ leads to smooth solution. Finally, by using the equivalence of homogeneous Sobolev norms we translate this to the stability result for $U_T$, thereby establishing Theorem \ref{maintheorem}.

\subsection{Notation and conventions}
 We write $a \lesssim b$ if there exists a constant $C >0$, such that $a \leq C b$ and we write $a \simeq b$ if $a \lesssim b$ and $b \lesssim a$. If the constant $C$ depends on some parameter $\varepsilon$, then we write $a \lesssim_{\varepsilon} b$.
We denote the open ball with radius $R>0$ in $\R^d$ by $B_R^d$ and drop the index $d$ if the dimension is clear from context. We also use the common Japanese bracket notation $\langle x \rangle := \sqrt{1+ |x|^2}$.
By $C^{\infty}(\R^d)$ and $\mc S(\R^d)$ we denote the space of smooth functions and the space of Schwartz functions respectively. By $C^{\infty}_{c}(\R^d)$ we denote the standard test space consisting of smooth and compactly supported functions. In case of radial functions we use a lower index $r$ as in $C^{\infty}_{r}(\R^d)$, $\mc S_r(\R^d)$, $C^{\infty}_{c,r}(\R^d)$. For convenience, we also write $C^{\infty}(\R^d)$, $C^{\infty}_{c}(\R^d)$ and $\mc S(\R^d)$ for sets of vector-valued functions whose every component belongs to that space. For a closed linear operator $(\mathcal{L}, \mathcal{D}(\mathcal{L}))$, we write $\rho(\mathcal{L})$ for the resolvent set, and $\sigma(\mathcal{L}) := \C \setminus \rho(\mathcal{L})$ for the spectrum. Given $\la \in \rho(\mc L)$, we use the following convention for the resolvent $R_{\mathcal{L}}(\lambda) := (\lambda - \mathcal{L})^{-1}$. For $f \in C^{\infty}_{c}(\R^d)$, we use the following definition of the Fourier transform
\begin{align*}
	\hat{f}(\xi) = \mathcal{F}f(\xi) := (2 \pi)^{-\frac{d}{2}} \int_{\R^d} e^{-i \xi \cdot x} f(x) dx, \quad \xi \in \R^d.
\end{align*}

\section{Formulation of the problem}

In this section we introduce similarity variables in which the self-similar blow up solution $\phi$ becomes a static solution.  Note that equation \eqref{Eq:CorHMHF} is a semilinear heat equation in dimension $d+2$. In the following we define $n:= d+2$, $n \geq 6$, and study the following Cauchy problem for $w(t,x):=v(t,|x|)$, $x \in \R^n$, and radial initial data close to $w_T$ with $T = 1$, where
\begin{equation}\label{BlowupSol_Trans}
	w_T(t,x)= \frac{1}{\sqrt{T-t}}\phi \left( \frac{|x|}{\sqrt{T-t}} \right), \quad 	\phi(\rho) = \frac{a}{\sqrt{\rho^2 + b}},
\end{equation}
with $a,b$ given in \eqref{ab}. More precisely, we consider for  radial functions $\varphi_0: \R^n \to \R$ the initial value problem
\begin{equation}\label{Orig_Cauchy_Problem}
	\begin{cases}
	~\displaystyle{ \left (\partial_t  - \Delta  \right )w(t,x) = \frac{n-3}{|x|^3}\Big(|x|w(t,x)-F \big(|x|w(t,x)\big)\Big), \quad t > 0,}\\[2mm]
	~w(0,\cdot)=\phi(|\cdot|) + \varphi_0,
	\end{cases}
\end{equation} 
where $F := g g'$ and $g$ is the warping function constructed in the proof of  Theorem \ref{geotheo}.

\subsection{Similarity variables}

Let $T>0$, for $t \in [0,T)$ and $x \in \R^n$ we define
\begin{align} \label{svariables}
\tau = \tau(t) := \ln \left ( \frac{T}{T-t} \right ) \quad \text{and} \quad y = y(t,x) := \frac{x}{\sqrt{T-t}}.
\end{align}
Consequently, the time interval $[0,T)$ is mapped into $[0,\infty)$. The partial derivative with respect to $t$ and the Laplacian become
\begin{align*}
\partial_t = \frac{e^{\tau}}{T} \left (\partial_{\tau} + \frac{1}{2} y \cdot \nabla_{y} \right ), \quad \Delta_x = \frac{e^{\tau}}{T} \Delta_{y}.
\end{align*}
With 
\begin{align} \label{psi}
\psi( \tau, y) := \sqrt{T} e^{-\frac{\tau}{2}} w \left (T-T e^{-\tau}, \sqrt{T} e^{-\frac{\tau}{2}} y \right ) 
\end{align}
we reformulate \eqref{Orig_Cauchy_Problem} as 

\begin{equation}\label{newequation}
	\begin{cases}
	~\displaystyle{ \left (\partial_{\tau} - \Delta_{y} + \Lambda  \right )\psi(\tau, y) = \frac{n-3}{|y|^3}\Big(|y|\psi(\tau, y)-F \big(|y|\psi(\tau, y)\big)\Big), \quad  \tau > 0,}\\[2mm]
	~\psi(0,\cdot)=\sqrt{T} \phi(\sqrt{T} |\cdot|) +\sqrt{T} \varphi_0( \sqrt{T} \cdot ),
	\end{cases}
\end{equation}
where we define the formal operator 
\begin{align} \label{Lambdaop}
[\Lambda f](y) := \frac{1}{2} \bigl( y \cdot \nabla f(y) + f(y) \bigl), \quad y \in \R^n,
\end{align}
acting on functions $f$ defined on $\R^n$. To study the evolution near $\phi$, we make the ansatz 
\[ \psi(\tau, \cdot) = \phi(|\cdot|) + \varphi(\tau, \cdot) \]
for a radial function $\varphi$ and write the evolution equation for the perturbation as 
\begin{equation} \label{Centralproblem}
\left\{
\begin{aligned} \partial_\tau \varphi (\tau, \cdot) &= L \varphi(\tau, \cdot) + \mathcal{N}(\varphi(\tau, \cdot)), \quad \tau > 0,\\
\varphi(0,\cdot) &= \mathcal{U}(\varphi_0, T),
\end{aligned}
\right.
\end{equation}
where $L := L_0 + L_1$, 
\begin{align} \label{sumoperators}
L_0 f: = \Delta f - \Lambda f, \quad L_1f  := V f,
\end{align}
with the radial potential $V$ given by
\begin{align} \label{potential}
V(y):= 3(n-3) \big (2 \alpha \phi^2(|y|)-5 \beta |y|^2 \phi^4(|y|) \big ), \quad y \in \R^n,
\end{align}
for $\alpha, \beta$ defined in \eqref{alphabeta}. The nonlinearity can be written as 
\begin{align} \label{Nonlinearity}
&[\mathcal{N}( \varphi(\tau, \cdot))](y)  \nonumber \\
&= \frac{(n-3)}{|y|^3} \Big ( F(|y| \phi (|y|)) + F'(|y| \phi (|y|)) |y| \varphi(\tau, y) - F \big (|y| \phi(|y|) + |y| \varphi(\tau, y) \big) \Big).
\end{align}
Finally, the initial condition for $\varphi$ is given by
\begin{align} \label{InitialU}
\varphi(0, \cdot )= \sqrt{T}  \phi(\sqrt{T} |\cdot|) + \sqrt{T} \varphi_0( \sqrt{T} \cdot)  - \phi(|\cdot |) =: \mc U(\varphi_0, T).
\end{align}

\subsection{Functional setup}
To create a suitable functional setup, we rely on the homogeneous Sobolev inner product
\begin{align*}
	\langle f,g \rangle_{\dot{H}^s(\R^n)} := \langle | \cdot |^s \mathcal{F} f, | \cdot |^s \mathcal{F} g \rangle_{L^2(\R^n)},
\end{align*}
where $f, g \in C^{\infty}_c(\R^d)$, $s \geq 0$, and $\mathcal{F}$ is the $n$-dimensional Fourier transform. This induces the homogeneous Sobolev norm on $C^{\infty}_c(\R^n)$
\begin{align}\label{Def:Sob_norm}
	\Vert f \Vert^2_{\dot{H}^s(\R^n)} := \langle f,f \rangle_{\dot{H}^s(\R^n)}.
\end{align}
As usual, the homogeneous Sobolev space $\dot{H}^s(\R^n)$ is defined as the completion of $C^{\infty}_c(\R^n)$ under the norm \eqref{Def:Sob_norm}.
Now, given $s,k \geq 0$ and $f, g \in C^{\infty}_c(\R^n)$, we define the inner product
\begin{align*}
	\langle f,g \rangle_{X_s^k(\R^n)} := \langle f,g \rangle_{\dot{H}^{s}(\R^n)}  + \langle f,g \rangle_{\dot{H}^{k}(\R^n)} ,
\end{align*}
which induces the norm $\Vert \cdot \Vert_{X_s^k(\R^n)}$ on $C^{\infty}_c(\R^n)$. This leads to the definition of the central space of the paper, $X_{s}^{k}(\R^n)$, which we define as the completion of the space of \textit{radial} test functions $C^{\infty}_{c,r}(\R^n)$ with respect to $\Vert \cdot \Vert_{X_s^k(\R^n)}$  for $(s,k)$ satisfying 
\begin{align}\label{sk-range}
\frac{n}{2} - 1 < s < \frac{n}{2} - 1 + \frac{1}{2(n-2)}, \quad \text{ and } k > n, \quad n \geq 6.
\end{align}
Furthermore, we set $\sigma(x) := e^{-|x|^2/4}$ for  $x \in \R^n$ and define a weighted $L^2$-space of radial functions
\begin{align}\label{L2sigma}
\mathcal{H} := \{ f \in L^2_{\sigma}(\R^n) : \text{$f$ is radial} \},
\end{align}
with induced norm $\Vert \cdot \Vert_{\mathcal{H}}$ coming from the inner product
\begin{align*}
\langle f, g \rangle_{\mathcal{H}} := \int_{\R^n} f(x) \overline{g(x)} \sigma(x) dx, \quad \text{for $f,g \in \mathcal{H}$}.
\end{align*}

\begin{lemma} \label{embeddinglemma}
For $(s,k)$ as in \eqref{sk-range}, the following embeddings hold
\[ X_s^k(\R^n) \hookrightarrow L^{\infty}(\R^n) \hookrightarrow \mathcal{H}.\] 
Furthermore, $X_s^k(\R^n)$ is closed under multiplication, i.e.,
\begin{align*}
\Vert f g \Vert_{X_s^k(\R^n)} \lesssim \Vert f \Vert_{X_s^k(\R^n)} \Vert g \Vert_{X_s^k(\R^n)},
\end{align*}
for all $f,g \in X_s^k(\R^n)$.
\end{lemma}

\begin{proof}
The embedding  $X_s^k(\R^n) \hookrightarrow L^{\infty}(\R^n)$ follows from the choice $s < \frac{n}{2} < k$ and  $ L^{\infty}(\R^n) \hookrightarrow \mathcal{H}$ is a consequence of the strong decay of the weight function $\sigma$. The details are provided in Appendix \ref{Embeddingproof}. The algebra property follows by the generalized Leibniz rule, see \cite{Grafakos}, Theorem 1, together with the embedding of $X_s^k(\R^n)$ into $L^{\infty}(\R^n)$. 
\end{proof}

\begin{lemma} \label{InXsk}
Let $(s,k)$ satisfy \eqref{sk-range}. If $f \in C^{\infty}_{r}(\R^n)$ satisfies
\begin{align*}
|\partial^{\alpha} f(x)| \lesssim \langle x \rangle^{-1-|\alpha|},
\end{align*}
for $\alpha \in \N_0^n$, $|\alpha| \leq k$ and all $x \in \R^n$, then $f \in X_s^k(\R^n)$.
\end{lemma}

\begin{proof}
The statement is proved by a standard approximation argument. For the interested reader we attached a detailed proof in Appendix \ref{Prooflemma42}.
\end{proof}

\section{Self-adjoint spectral theory}

In this section, we investigate the linear operator $L$ that appears in our main problem \eqref{Centralproblem}. We first determine the spectrum of $L$ in the self-adjoint setting and use this to show exponential decay (on a suitable subspace) of the semigroup generated by $L$ in $X_s^k(\R^n)$.
The splitting $L = L_0 + L_1$ allows us to study the free operator $L_0$ first and then extend our results to the full case by perturbation.

It is easy to see that $L$ together with the domain $\mathcal{D}(L) := C^{\infty}_{c,r}(\R^n)$ is an unbounded, densely defined, symmetric operator on $\mathcal{H}$. We have the following result.

\begin{proposition} \label{PropoQ}
The operator $(L,\mathcal{D}(L))$ is closable in $\mc H$ and the closure $\mathcal{L}: \mathcal{D}(\mathcal{L}) \subseteq \mathcal{H} \rightarrow \mathcal{H}$ is self-adjoint, has compact resolvent and generates a strongly continuous semigroup $(S_0(\tau))_{\tau \geq 0}$ of bounded operators on $\mathcal{H}$. For the spectrum of $\mathcal{L}$, which consists only of eigenvalues, we have 
\begin{align*}
\sigma( \mathcal{L} ) \subseteq (- \infty, 0) \cup \{ 1 \}.
\end{align*}
The spectral point $\lambda =1$ is a simple eigenvalue with the normalized eigenfunction 
\begin{align}\label{Gauge}
 G(x) = \frac{\varrho(|x|)}{\Vert \varrho \Vert_{\mathcal{H}}}, \quad  \varrho(|x|) = (|x|^2 + b)^{-\frac{3}{2}}
\end{align}
with constant $b$ from \eqref{ab}.
\end{proposition}

\begin{proof}
We first show that the operator $L_0$ with $\mathcal{D}(L_0) := \mathcal{D}(L)$ is closable with closure $\mathcal{L}_0 : \mathcal{D}(\mathcal{L}_0) \subseteq \mathcal{H} \rightarrow \mathcal{H}$ being self-adjoint, having compact resolvent and generating a strongly continuous semigroup  on $\mathcal{H}$.  This semigroup is explicitly given by
\begin{align} \label{S0}
[S_0(\tau) f](x) = e^{-\frac{\tau}{2}} \bigl( H_{\kappa(\tau)} \ast f \bigl) (e^{- \frac{\tau}{2}}x), \quad x \in \R^n,
\end{align}
where $H_{\kappa(\tau)}(x) = e^{-\frac{|x|^2}{4 \kappa(\tau)}}(4 \pi \kappa(\tau))^{-\frac{n}{2}}$ and $\kappa(\tau) := 1- e^{-\tau}$. This follows from the unitary equivalence of $L_0$ to the one-dimensional Schrödinger operator 
\begin{align} \label{A0}
[A_0 u](\rho) = -u^{\prime \prime}(\rho) + q(\rho) u(\rho),  \quad \rho \in \R^+,
\end{align}
with
\begin{align*}
q(\rho) := \frac{\rho^2}{16} + \frac{(n-3)(n-1)}{4 \rho^2} - \frac{n-2}{4},
\end{align*}
and domain $\mathcal{D}(A_0)= U^{-1} \mathcal{D}(L_0)$, where the unitary operator $U$ is defined as
\begin{align*}
U: L^2(\R^+) \rightarrow \mathcal{H}, \quad u \mapsto Uu = |S^{n-1}|^{-\frac{1}{2}} | \cdot |^{-\frac{n-1}{2}} e^{\frac{| \cdot |^2}{8}} u(| \cdot |).
\end{align*}
It is easy to see that $-L_0 = U A_0 U^{-1}$. Using the properties of $q$, standard results imply that the unique self-adjoint extension of $A_0$ is given by the maximal operator $\mc A_0: \mc D(\mc A_0) \subseteq L^2(\R^+)\rightarrow L^2(\R^+) $,
\begin{align} \label{Def:A0}
\mc D(\mc A_0) := \{ u \in L^2(\R^+): u, u' \in AC_{\mathrm{loc}}(\R^+), A_0 u \in L^2(\R^+) \};
\end{align}
see \cite{YangMills}, Lemma 3.1 for the details. Moreover, $\mc A_0$ has compact resolvent and is bounded from below which implies the same for the self-adjoint operator $\mc L_0 := - U \mc A_0 U^{-1}$, $\mc D(\mc L_0) := U \mc D(\mc A_0)$. Consequently, $\mathcal{L}_0$ generates a  strongly continuous semigroup on $\mathcal{H}$. Using the transformation for the time variable  given by \eqref{svariables} for $T=1$, we infer that $H_{\kappa(\tau)}$ is just the standard heat kernel in self-similar coordinates. By considering the definition of a semigroup generator one can check that $\mathcal{L}_0$ generates the semigroup $(S_0(\tau))_{\tau \geq 0}$.

To study the operator $L=L_0+L_1$, we start with the observation that the potential $V$ given by \eqref{potential} is bounded, i.e., $\Vert V \Vert_{L^{\infty}(\R^n)} < \infty$, which implies the boundedness of the operator $L_1$ on $\mathcal{H}$. Therefore, $L$ is closable and the closure is given by $\mathcal{L} = \mathcal{L}_0 + L_1$ with domain $\mathcal{D}(\mathcal{L}) = \mathcal{D}(\mathcal{L}_0)$. The Kato-Rellich theorem (see, e.g., \cite{Teschl}, Theorem 6.4), shows the self-adjointness of $\mathcal{L}$ and it similarly follows that $\mathcal{L}$ has compact resolvent. Furthermore, the Bounded Perturbation Theorem (see, e.g., \cite{Engel}, p.~158, Theorem 1.3) yields the fact that $\mathcal{L}$ is the generator of a strongly continuous semigroup $(S(\tau))_{\tau \geq 0}$ on $\mathcal{H}$.

Next we investigate the spectrum of $\mathcal{L}$ via the properties of $\mathcal{A} : \mathcal{D}(\mathcal{A}) \subseteq L^2(\R^+) \rightarrow L^2(\R^+)$ defined by $\mathcal{D}(\mathcal{A}) = \mathcal{D}(\mathcal{A}_0)$, $\mathcal{A}u = \mathcal{A}_0 u - L_1u$. This operator is unitary equivalent to $- \mathcal{L}$ via the map $U$. Thus it suffices to study the spectrum of $\mathcal{A}$. To show $\sigma(\mathcal{A}) \subseteq \{ -1 \} \cup (0, \infty)$, we first observe that 
\begin{align*}
\varrho_A (\rho) := e^{-\frac{\rho^2}{8}} \rho^{\frac{n-1}{2}} (\rho^2 +b)^{-\frac{3}{2}}, \quad \rho >0,
\end{align*}
belongs to $\mathcal{D}(\mathcal{A})$ and satisfies $(-Id - \mathcal{A})\varrho_A =0$, which means $-1 \in \sigma(\mathcal{A})$. The unitary equivalence of $-\mathcal{L}$ and $\mathcal{A}$ implies that 
\begin{align*}
(U \varrho_A)(x) = |S^{n-1}|^{-\frac{1}{2}} (|x|^2 +b)^{-\frac{3}{2}}, \quad x \in \R^n,
\end{align*}
is an eigenfunction of $\mathcal{L}$ to the eigenvalue $\lambda =1$, i.e., $1 \in \sigma(\mathcal{L})$ and $G$ as defined in \eqref{Gauge} is a normalized eigenfunction. By inspection of the spectral ODE it follows that the geometric eigenspace is indeed one-dimensional. 

To show that $\sigma(\mathcal{A}) \setminus \{ -1 \}$ is contained in $(0 , \infty)$, we factorize $\mathcal{A}+1$ to get a new self-adjoint operator that is isospectral with $\mathcal{A}$ except $\lambda = -1$. We write $\mathcal{A} = A^+ A^- -1$ for suitable operators $A^+,A^-$ such that $\text{ker}(A^-) = \text{span}(\varrho_A)$. Explicitly,
\begin{align*}
A^+ =  - \partial_\rho - \frac{\varrho_A^{\prime}}{\varrho_A} , \quad \text{and} \quad A^- = \partial_\rho - \frac{\varrho_A^{\prime}}{\varrho_A}.
\end{align*}
We define a self-adjoint operator corresponding to $A^- A^+  -1 $,
\begin{align*} 
\mathcal{A}_S : \mathcal{D}(\mathcal{A}_S) \subseteq L^2(\R^+) \rightarrow L^2(\R^+), \quad [\mathcal{A}_S u ](\rho) = - u^{\prime \prime} (\rho) + \Big( \frac{n^2-1}{4 \rho^2} + Q(\rho) \Big) u(\rho),
\end{align*}
where $Q$ is given by
\begin{align*} 
Q(\rho) = \frac{\rho^2}{16} - \frac{n}{4} +1 +\frac{3b(5a^4 \beta (n-3)-6)}{(\rho^2+b)^2} + \frac{3(2-b-(n-3)(5a^4 \beta -2 a^2 \alpha +2))}{\rho^2+b}, 
\end{align*}
for $\rho > 0$. 
The constants appearing here are those from \eqref{alphabeta} and \eqref{ab}.
In order to exclude non-positive eigenvalues of $\mathcal{A}_S$, we apply a GGMT type integral criterion as  stated in \cite{YangMills}, Theorem A.1. More precisely, we define for $n \in \{6, 7, 8, 9 \}$ and $p > 1$,
\begin{align} \label{Bnp}
B(n,p) := c(n,p) \int_0^{\infty} \rho^{2p-1} |Q_{-}(\rho)|^p d\rho, \quad c(n,p):= \frac{(p-1)^{p-1} \Gamma(2p)}{n^{2p-1} p^p \Gamma(p)^2},
\end{align}
where $Q_{-}(\rho) := \min \{ 0, Q(\rho) \}$ for $\rho>0$. We show that for every $n$ we find a suitable  $p$ such that $B(n,p) < 1$.
It is easy to see that the potential $Q$ has exactly one zero $\rho^* = \rho^*(n)>0$ and satisfies $Q|_{(0, \rho^*)} <0$ and $Q|_{(\rho^*,\infty)} >0$. For $n=6$ and $p=2$, the integrand in \eqref{Bnp} is a simple rational function. Thus,  one finds that 
\[ B(6,2) < c(6,2) \int_0^{\frac{21}{5}} \rho^3 Q(\rho)^2 < \frac{4}{5},\] 
where the last integral can be computed explicitly. 
For $n \in \{7, 8, 9 \}$ we proceed analogously and choose for $(n,p)$ the pairs $(7,2), (8,4), (9,4)$. Theorem A.1 in \cite{YangMills} now implies that the spectrum of $\mathcal{A}_S$ is contained in $[0, \infty)$ and that zero is not an eigenvalue. As $\mathcal{A}_S$ is isospectral with $\mathcal{A}$ modulo $\lambda =-1$, we infer that  $\sigma(\mathcal{A}) \subseteq \{ -1 \} \cup (0, \infty)$ and by unitary equivalence it follows that $\sigma(\mathcal{L}) \subseteq (-\infty, 0) \cup \{1 \}$.
\end{proof}

\begin{remark}
The eigenvalue $\lambda = 1$ of $\mathcal{L}$  is due to the time translation symmetry of the problem and will be controlled later on by variation of the blowup time $T$ in \eqref{Centralproblem}.
\end{remark}

In the following, we define the orthogonal projection onto the unstable mode $G$,
\begin{align} \label{Projection}
\mathcal{P}f := \langle f, G \rangle_{\mathcal{H}} G, \quad \text{for $f \in \mathcal{H}$}.
\end{align}
By the spectral structure of $\mc L$, the linear evolution decays exponentially on the invariant subspace $\mathrm{ker} \mathcal~{\mc P}$. Moreover, an even stronger result holds in terms of the graph norms corresponding to fractional powers of the positive operator $1- \mathcal{L}$. 
More precisely, for $k \in \N_0$ and $f \in \mathcal{D}((1-\mathcal{L})^{\frac{k}{2}})$ we define
\begin{align} \label{Graphnorm}
\Vert f \Vert_{\mathcal{G}((1-\mathcal{L})^{\frac{k}{2}})} := \Vert f \Vert_{\mathcal{H}} + \Vert (1-\mathcal{L})^{\frac{k}{2}} f \Vert_{\mathcal{H}},
\end{align}
as the graph norm of order $k$. We note that $C^{\infty}_{c,r}(\R^n)$ is a core of $ (1-\mathcal{L})^{\frac{k}{2}} $.

\begin{corollary} \label{Estimatesingraph}
There exists $\omega_0 >0$ such that for all $k \in \N_0$ and $\tau \geq 0$ we have
\begin{align} \label{HGGraphnorm}
\Vert S(\tau) (1- \mathcal{P})f \Vert_{\mathcal{G}((1-\mathcal{L})^{\frac{k}{2}})} \leq e^{- \omega_0 \tau} \Vert  (1- \mathcal{P})f \Vert_{\mathcal{G}((1-\mathcal{L})^{\frac{k}{2}})},
\end{align}
where $f \in \mathcal{D}((1-\mathcal{L})^{\frac{k}{2}})$.
\end{corollary}

\begin{proof}
The proof is a direct consequence of Proposition \ref{PropoQ} together with the commutating properties of $(1-\mathcal{L})^{\frac{k}{2}}$ for $k \in \N_0$; see \cite{YangMills}, Proposition 3.6, for the details.
\end{proof}

\section{The linear evolution on $X_s^k(\R^n)$} 

In this section, we consider the properties of the semigroup constructed in Proposition \ref{PropoQ} when being restricted to $X_s^k(\R^n)$. We state some technical results first. 

\begin{lemma} \label{graphnormestimate}
For $j \in \{ 0,...,k \}$, we have $X_s^k(\R^n) \subseteq \mathcal{D}((1-\mathcal{L})^{\frac{j}{2}})$ and 
\begin{align}
\Vert (1-\mathcal{L})^{\frac{j}{2}} f \Vert_{\mathcal{H}} \lesssim \Vert f \Vert_{X_s^k(\R^n)},
\end{align}
for all $f \in X^k_s(\R^n)$.
\end{lemma}

\begin{proof}
One finds that there exist smooth, radial and polynomially bounded functions $w_{\alpha}$ satisfying
\begin{align*}
\Vert (1-\mathcal{L})^{\frac{j}{2}} f \Vert_{\mathcal{H}} \lesssim \sum_{|\alpha| \leq j} \Vert w_{\alpha} \partial^{\alpha} f \Vert_{\mathcal{H}},
\end{align*}
for all $f \in C^{\infty}_{c,r}(\R^n)$. Exploiting the exponential decay of the weight function $\sigma$, we obtain
\begin{align*}
\Vert w_{\alpha} \partial^{\alpha} f \Vert_{\mathcal{H}} \lesssim \Vert | \cdot |^{-s_{\alpha}} \partial^{\alpha} f \Vert_{L^2(\R^n)} \lesssim \Vert f \Vert_{\dot{H}^{s_{\alpha}+|\alpha|}} \lesssim \Vert f \Vert_{X_s^k(\R^n)},
\end{align*}
by Hardy's inequality, where $s_{\alpha} = \max \{ s- |\alpha|, 0 \}$. Note that the last inequality follows by $s \leq s_{\alpha} + |\alpha| \leq k$. For general $f \in X_s^k(\R^n)$ we use the density of $C^{\infty}_{c,r}(\R^n)$ in $X_s^k(\R^n)$ and the closedness of $(1-\mathcal{L})^{\frac{j}{2}}$.
\end{proof}

\begin{lemma} \label{EstimateinBR}
Let $j \in \N_0$ and $R>0$. Then 
\begin{align} \label{estimateinBR}
\Vert \partial^{\alpha} f \Vert_{L^2(B^n_{R})} \lesssim \sum_{m=0}^j \Vert f \Vert_{\mathcal{G}((1-\mathcal{L})^{\frac{m}{2}})},
\end{align}
for all $f \in C^{\infty}_{c,r}(\R^n)$ and all $\alpha \in \N_0^n$ with $|\alpha| = j$.
\end{lemma}

\begin{proof}
Let $R>0$. We define the operator
\begin{align*}
Bf :=\varrho  \left  ( \frac{\tilde f}{\varrho } \right )',
\end{align*}
for $f \in C^{\infty}_{c,r}(\R^n)$, where $f = \tilde f(|\cdot|)$, and $x \in \R^n$, with $\varrho$ denoting the eigenfunction of $\mathcal{L}$ to the eigenvalue $\lambda=1$; see Proposition \ref{PropoQ}. The formal adjoint operator in $\mathcal{H}$ is given by
\begin{align*}
B^* f = - \frac{\tilde f'}{\mu} ,
\end{align*}
for $\mu(\rho) :=  e^{-\frac{\rho}{4}} \rho^{n-1} \varrho(\rho)$ 
and we can write  $(1- L) f =( B^*B \tilde f) (|\cdot|)$.  For $f \in C^{\infty}_{c,r}(\R^n)$ this implies
\begin{align*}
\Vert (Bf )(|\cdot|)\Vert_{\mathcal{H}}^2 &  = \langle Bf(|\cdot|) , Bf(|\cdot|) \rangle_{\mathcal{H}} = \langle (B^* Bf)(|\cdot|) , f) \rangle_{\mathcal{H}}  \\
& = \langle (1-\mathcal{L}) f , f \rangle_{\mathcal{H}} = \Vert (1-\mathcal{L})^{\frac{1}{2}} f \Vert_{\mathcal{H}}^2.
\end{align*}
As inequality \eqref{estimateinBR} is trivial for $|\alpha|=0$ due to the decay of the exponential weight function $\sigma$, we first consider the case $|\alpha|=1$ and then show the estimate by induction. For $|\alpha| = 1$ and $i \in \{1,...,n \}$,  we have
\begin{align*}
\Vert \partial_i f \Vert^2_{L^2(B_R)} &= \int_{B_R} \frac{x_i^2}{|x|^2} |\tilde f^{\prime}(|x|)|^2 dx \leq \int_{B_R} \Bigg( |(Bf)(|x|)|^2 + |f(x)|^2 \left( \tfrac{\varrho^{\prime}(|x|)}{\varrho(|x|)} \right)^2 \Bigg ) dx\\
&\leq e^{\frac{R^2}{4}} \int_{B_R} e^{-\frac{|x|^2}{4}} \Bigg(|(Bf)(|x|)| + |f(x)|^2 \left( \tfrac{\varrho^{\prime}(|x|)}{\varrho(|x|)} \right)^2 \Bigg ) dx\\
&\leq C_R \bigl( \Vert Bf \Vert^2_{\mathcal{H}} + \Vert f \Vert^2_{\mathcal{H}} \bigl),
\end{align*}
which implies
\begin{align*}
\Vert \partial^{\alpha} f \Vert_{L^2(B_R)} \lesssim_R \Vert f \Vert_{\mathcal{G}((1-\mathcal{L})^{\frac{1}{2}})},
\end{align*}
for all $f \in C^{\infty}_{c,r}(\R^n)$ and $\alpha \in \N_0^n$ with $|\alpha|=1$. To use induction we need the following estimate 
\begin{align}\label{Est:D}
\Vert \partial^{\alpha} f \Vert_{L^2(B_R)} \lesssim \sum_{i=0}^{|\alpha|} \Vert D^j f \Vert_{L^2(B_R)},
\end{align} 
which holds for all $f \in C^{\infty}_r(B_R^n)$ and $\alpha \in \N_0^n$ (see \cite{KellerSegel}, Lemma A.1), where the operators $D^j$ are defined as follows: If $j$ is even, one has  $D^j  f := \Delta_r^{\frac{j}{2}} \tilde f$ and if $j$ is odd $D^j f := \bigl( \Delta_r^{\frac{j-1}{2}} \tilde f \bigl)^{\prime}$, where $\Delta_r$ denotes the radial Laplacian on $\R^n$. Note that Lemma A.1 in \cite{KellerSegel} is stated in dimension $n=5$ but can be straightforwardly adapted to higher dimensions. Furthermore, we use the commutator relation
\begin{align*}
D^j \Lambda = \Lambda D^j + \frac{j}{2} D^j,
\end{align*}
for $j \in \N_0$ with the formal operator $\Lambda$ from \eqref{Lambdaop}. Now assume that the inequality \eqref{estimateinBR} holds up to some $j \in \N$. Then 
\begin{align*}
\Vert D^{j+1} f \Vert_{L^2(B_R)} &= \Vert D^{j-1} (D^2 f) \Vert_{L^2(B_R)} = \Vert D^{j-1} ((\mathcal{L}-1+1 + \Lambda - V)f) \Vert_{L^2(B_R)}\\
&\lesssim \Vert D^{j-1} ((1-\mathcal{L})f) \Vert_{L^2(B_R)} + \Vert D^{j-1} ((\Lambda+1)f) \Vert_{L^2(B_R)} + \Vert D^{j-1} (Vf) \Vert_{L^2(B_R)}\\
&\lesssim \Vert D^{j-1} ((1-\mathcal{L})f) \Vert_{L^2(B_R)} + \Vert \bigl( \Lambda + \tfrac{j+1}{2} \bigl) D^{j-1} f \Vert_{L^2(B_R)} + \Vert D^{j-1} (Vf) \Vert_{L^2(B_R)}\\
&\lesssim_{R} \sum_{\substack{\alpha \in \N_0^n \\ |\alpha| \leq j-1}} \Vert \partial^{\alpha} ((1-\mathcal{L})f) \Vert_{L^2(B_R)} + \sum_{\substack{\alpha \in \N_0^n \\ |\alpha| \leq j}} \Vert \partial^{\alpha} f \Vert_{L^2(B_R)}\\
&\lesssim_{R} \sum_{l=0}^{j-1} \Vert (1-\mathcal{L}) f \Vert_{\mathcal{G}((1-\mathcal{L})^{\frac{l}{2}})} + \sum_{l=0}^j \Vert f \Vert_{\mathcal{G}((1-\mathcal{L})^{\frac{l}{2}})}\\
&\lesssim_R \sum_{l=0}^{j+1} \Vert f \Vert_{\mathcal{G}((1-\mathcal{L})^{\frac{l}{2}})},
\end{align*}
for all $f \in C^{\infty}_{c,r}(\R^n)$, which, together with \eqref{Est:D} implies the claim.
\end{proof}

\subsection{The semigroup on $X_s^k(\R^n)$}\label{Sec:Semigroup_X}

\begin{proposition} 
The restriction of $(S(\tau))_{\tau \geq 0}$ to $X_s^k(\R^n)$ defines a strongly continuous one-parameter semigroup $(S_{X_s^k}(\tau))_{ \tau \geq 0 }$ on $X_s^k(\R^n)$. Its generator is given by the part of $\mathcal{L}$ in $X_s^k(\R^n)$, namely
\begin{align*}
\mathcal{L}_{X_s^k} f := \mathcal{L} f, \quad \mathcal{D}(\mathcal{L}_{X_s^k}) := \{ f \in \mathcal{D}(\mathcal{L}) \cap X_s^k(\R^n) : \mathcal{L} f \in X_s^k(\R^n) \}.
\end{align*}
Furthermore, the set of radial Schwartz functions $\mc S_{r}(\R^n)$ is a core of $\mathcal{L}_{X_s^k}$.
\end{proposition}

\begin{proof}
We first consider the free semigroup $S_0(\tau)$ and show that $X_s^k(\R^n)$ is invariant under its action. Let $f \in C^{\infty}_{c,r}(\R^n)$ and $\tau >0$. By the explicit representation \eqref{S0} we obtain
\begin{align}\label{MapPropS0}
\Vert | \cdot |^s \mathcal{F} ( S_0(\tau) f) \Vert_{L^2(\R^n)} \lesssim e^{\frac{1}{2}(\frac{n}{2}-1-s)\tau} \Vert | \cdot |^s \mathcal{F} f \Vert_{L^2(\R^n)},
\end{align}
where we used that the heat kernel $H_{\kappa(\tau)}$ is an element of $L^1(\R^n)$ for $\tau >0$ with norm equal to $1$. The same holds for $k$ instead of $s$. This implies the following bound for the operator norm
\begin{align*}
\Vert S_0(\tau) \Vert_{\mathcal{L}(X_s^k(\R^n))} \lesssim \max \{e^{\frac{1}{2}(\frac{n}{2}-1-s)\tau}, e^{\frac{1}{2}(\frac{n}{2}-1-k)\tau} \} = e^{\frac{1}{2}(\frac{n}{2}-1-s)\tau}.
\end{align*} 
In a similar way we get the equality
\begin{align*}
\Vert | \cdot |^s \mathcal{F} ( S_0(\tau) f - f) \Vert_{L^2(\R^n)} = \Vert | \cdot |^s (e^{\frac{\tau}{2}(n-1)} e^{-(e^{\tau}-1)| \cdot |^2} \mathcal{F}(f)(e^{\frac{\tau}{2}} \cdot) - \mathcal{F}f) \Vert_{L^2(\R^n)},
\end{align*}
by the explicit form of the heat kernel on the Fourier side. Again, the same holds for $k$ instead of $s$. An application of the dominated convergence theorem shows 
\[\lim_{\tau \rightarrow 0^+}( S_0(\tau) f - f) = 0\]
in $X_s^k(\R^n)$ for all $f \in C^{\infty}_{c,r}(\R^n)$ and hence for all $f \in X_s^k(\R^n)$. Together with the embedding $X_s^k(\R^n) \hookrightarrow \mathcal{H}$ by Lemma \ref{embeddinglemma} all conditions for $( S_0(\tau))_{\tau \geq 0}$ to be a strongly continuous semigroup on $X_s^k(\R^n)$ are satisfied. A standard result from semigroup theory (see \cite{Engel}, p.~60, Proposition 2.3) yields that the restriction of $\mathcal{L}_0$ to $X_s^k(\R^n)$ denoted by $\mathcal{L}_0|_{X_s^k}$ with domain $\mathcal{D}(\mathcal{L}_0|_{X_s^k}) = \{ f \in \mathcal{D}(\mathcal{L}_0) \cap X_s^k(\R^n) : \mathcal{L}_0f \in X_s^k(\R^n) \}$ generates the restricted semigroup $( S_0(\tau)|_{X_s^k})_{\tau \geq 0}$. 

Since $X_s^k(\R^n)$ is closed under multiplication, see Lemma \ref{embeddinglemma},  and  $V \in X_s^k(\R^n)$, we infer that the operator $L_1$ from \eqref{sumoperators} is bounded on $X_s^k(\R^n)$. The Bounded Perturbation Theorem implies that $\mathcal{L}_{X_s^k}$ generates a strongly continuous semigroup $(S_{X_s^k}(\tau))_ {\tau \geq 0}$ on $X_s^k(\R^n)$. That this coincides indeed with the restriction of $(S(\tau))_ {\tau \geq 0 }$ to $X_s^k(\R^n)$ can be seen as in  \cite{YangMills}, Proposition 3.14. Finally, since the set of radial Schwartz functions $\mc S_r(\R^n) \subset X_s^k(\R^n)$ is dense and left invariant under the action of $ S_0(\tau) $ for $\tau \geq 0$ we infer that $\mc S_r(\R^n)$ is a core of $\mathcal{L}_0|_{X_s^k}$ and hence for the full operator $\mathcal{L}_{X_s^k}$.
\end{proof}

\begin{lemma}
The restriction of the projection operator $\mathcal{P}$ defined by \eqref{Projection} to $X_s^k(\R^n)$ induces a (non-orthogonal) projection $\mathcal{P}_{X_s^k}$ on $X_s^k(\R^n)$,
\begin{align*}
\mathcal{P}_{X_s^k} f = \langle f , G \rangle_{\mathcal{H}} G, \quad \text{for $f \in X_s^k(\R^n)$},
\end{align*}
which commutes with the operator $\mathcal{L}_{X_s^k}$ and the semigroup $S_{X_s^k}(\tau)$ for all $\tau \geq 0$. The kernel of $\mathcal{P}_{X_s^k}$ is given by
\begin{align*}
\ker \mathcal{P}_{X_s^k} = \{ f \in X_s^k(\R^n) : \langle f , G \rangle_{\mathcal{H}} = 0 \}.
\end{align*}
\end{lemma}

\begin{proof}
By Lemma \ref{InXsk} and the fact that $G = C \phi(|\cdot|)^3$ for some constant $C \in \R$ we infer that $G \in X_s^k(\R^n)$. For $f \in X_s^k(\R^n)$ we have
\begin{align*}
\Vert \mathcal{P}_{X_s^k} f \Vert_{X_s^k} = | \langle f , G \rangle_{\mathcal{H}} | \Vert G \Vert_{X_s^k} \leq \Vert f \Vert_{\mathcal{H}} \Vert G \Vert_{\mathcal{H}} \Vert G \Vert_{X_s^k} \lesssim \Vert f \Vert_{X_s^k},
\end{align*}
by Lemma \ref{embeddinglemma}. That $\mathcal{P}_{X_s^k}$ commutes with $\mathcal{L}_{X_s^k}$ follows directly by $\mathcal{G}$ being the eigenfunction of $\mathcal{L}$ to the eigenvalue $\lambda =1$ and the self-adjointness of $\mathcal{L}$ in $\mathcal{H}$. That $\mathcal{P}_{X_s^k}$ commutes with $S_{X_s^k}(\tau)$ for all $\tau \geq 0$ can be easily seen by again using the aforesaid properties.
\end{proof}
The main goal of this section is to show that the restricted semigroup $(S_{X_s^k}(\tau))_{\tau \geq 0}$ exhibits exponential decay on the stable subspace $\ker \mathcal{P}_{X_s^k}$  in $X_s^k(\R^n)$. For this, we need the following estimate for the potential. 

\begin{lemma} \label{EstimateVf}
Let $R > 0$. There are constants $C_R, C, \varepsilon_V >0$ such that
\begin{align*}
\Vert V f \Vert_{X_s^k(\R^n)} \leq C_R \sum_{j=0}^k \Vert f \Vert_{\mathcal{G}((1-\mathcal{L})^{\frac{j}{2}})} + \frac{C}{R^{\varepsilon_V}} \Vert f \Vert_{X_s^k(\R^n)},
\end{align*}
for all $f \in X_s^k(\R^n)$.
\end{lemma}

\begin{proof}
To show the claimed estimate we use that for all $f \in C^{\infty}_{c}(\R^n)$ and $s$ as in \eqref{sk-range} we have
\begin{align}\label{WspEmb}
\| f \|_{\dot H^s(\R^n)} \simeq  \| |\cdot|^{-(n- \lceil s \rceil + s)} \ast (|\nabla|^{\lceil s \rceil} f) \|_{L^2(\R^n)}  \lesssim \|  |\nabla|^{\lceil s \rceil} f \|_{L^p(\R^n)}  = :\| f \|_{\dot{W}^{\lceil s \rceil, p}(\R^n)}
\end{align}
for $\frac{1}{p} = \frac{1}{2} + \frac{\lceil s \rceil - s}{n} $ by the Hardy-Littlewood-Sobolev inequality (see, e.g., \cite{Tao}, p.~335). This allows us to take integer derivatives of the product $Vf$ instead of fractional ones appearing in the $\dot{H}^s(\R^n)$ norm. In particular, we can use equivalence of norms (see, \cite{Tao}, p.~331), and the standard Leibniz rule to infer that 
\begin{align*}
\Vert Vf \Vert_{\dot{H}^s(\R^n)} &  \lesssim \Vert Vf \Vert_{\dot{W}^{\lceil s \rceil, p}(\R^n)} \lesssim \sum_{\substack{\alpha \in \N_0^n \\ |\alpha| = \lceil s \rceil}} \Vert \partial^{\alpha} (Vf) \Vert_{L^p(\R^n)}  \\
& \lesssim \sum_{\substack{\alpha \in \N_0^n \\ |\alpha| = \lceil s \rceil}} \sum_{\substack{\beta \in \N_0^n \\ \beta \leq \alpha}} \binom{\alpha}{\beta} \Vert \partial^{\beta} V \partial^{\alpha-\beta} f \Vert_{L^p(\R^n)},
\end{align*}
for all $f \in C^{\infty}_{c,r}(\R^n)$. Thus, it suffices to estimate $\partial^{\beta} V \partial^{\alpha-\beta} f$ in $L^p(\R^n)$ for $\alpha, \beta \in \N_0^n$ with $\beta \leq  \alpha $. For this we split
\begin{align*}
\Vert \partial^{\beta} V \partial^{\alpha-\beta} f \Vert_{L^p(\R^n)} \leq \Vert \partial^{\beta} V \partial^{\alpha-\beta} f \Vert_{L^p(B_R)} + \Vert \partial^{\beta} V \partial^{\alpha-\beta} f \Vert_{L^p(B_R^c)},
\end{align*}
and use the decay $|\partial^{\beta} V(x)| \lesssim \langle x \rangle^{-2-|\beta|}$ for all $\beta \in \N_0^n$ and $x \in \R^n$. In $B_R$ we obtain by H\"older's inequality with $\frac{1}{p} = \frac{1}{2} + \frac{1}{q}$ and Lemma \ref{EstimateinBR}
\begin{align*}
\Vert \partial^{\beta} V \partial^{\alpha-\beta} f \Vert_{L^p(B_R)} &\leq \Vert \partial^{\beta} V \Vert_{L^q(B_R)} \Vert \partial^{\alpha-\beta} f \Vert_{L^2(B_R)} \leq C_R \sum_{j=0}^{|\alpha|-|\beta|} \Vert f \Vert_{\mathcal{G}((1-\mathcal{L})^{\frac{j}{2}})},
\end{align*}
for some constant $C_R >0$. In $B_R^c$ we use again H\"older's inequality and exploit the decay of $V$ to get
\begin{align*}
\Vert \partial^{\beta} V \partial^{\alpha-\beta} f \Vert_{L^p(B_R^c)} & \leq C \Vert |\cdot|^{-2+m} \Vert_{L^q(B_R^c)} \Vert |\cdot|^{-|\beta|-m} \partial^{\alpha-\beta} f \Vert_{L^2(B_R^c)}\\
&\leq \frac{C}{R^{\varepsilon_V}} \Vert f \Vert_{\dot{H}^{|\alpha|+m}(\R^n)} \leq \frac{C}{R^{\varepsilon_V}} \Vert f \Vert_{X_s^k(\R^n)},
\end{align*}
for some $C = C(m) >0$ and $\varepsilon_V = \varepsilon(m) >0$, where $m = \max \{ s- |\alpha|, - |\beta| \}$. Note that the choice of $m$ allows for Hardy's inequality and the embedding $X_s^k(\R^n) \hookrightarrow \dot{H}_r^{|\alpha|+m}(\R^n)$. For the exponent $k$ we obtain
\begin{align*}
\Vert Vf \Vert_{\dot{H}^k(\R^n)} \simeq \sum_{\substack{\alpha \in \N_0^n \\ |\alpha|=k }} \Vert \partial^{\alpha} f \Vert_{L^2(\R^n)} \leq \sum_{\substack{\alpha \in \N_0^n \\ |\alpha| = k}} \sum_{\substack{\beta \in \N_0^n \\ \beta \leq \alpha}} \binom{\alpha}{\beta} \Vert \partial^{\beta} V \partial^{\alpha-\beta} f \Vert_{L^2(\R^n)},
\end{align*}
for all $f \in C^{\infty}_{c,r}(\R^n)$, where 
\begin{align*}
\Vert \partial^{\beta} V \partial^{\alpha-\beta} f \Vert_{L^2(\R^n)} \leq \Vert \partial^{\beta} V \partial^{\alpha-\beta} f \Vert_{L^2(B_R)} + \Vert \partial^{\beta} V \partial^{\alpha-\beta} f \Vert_{L^2(B_R^c)}.
\end{align*}
Again the behaviour of $V$ together with Lemma \ref{EstimateinBR} implies
\begin{align*}
\Vert \partial^{\beta} V \partial^{\alpha-\beta} f \Vert_{L^2(B_R)} \leq C_R \sum_{j=0}^{|\alpha|-|\beta|} \Vert f \Vert_{\mathcal{G}((1-\mathcal{L})^{\frac{j}{2}})} \leq C_R \sum_{j=0}^{k} \Vert f \Vert_{\mathcal{G}((1-\mathcal{L})^{\frac{j}{2}})},
\end{align*} 
for some constant $C_R >0$. In $B_R^c$ we infer
\begin{align*}
\Vert \partial^{\beta} V \partial^{\alpha-\beta} f \Vert_{L^2(B_R^c)} &\leq C \Vert | \cdot |^{-2 - |\beta|} \partial^{\alpha-\beta} f \Vert_{L^2(B_R^c)} \\
&\leq \frac{C}{R^2} \Vert | \cdot |^{-|\beta|} \partial^{\alpha-\beta} f \Vert_{L^2(B_R^c)} \\
&\leq \frac{C}{R^2} \Vert | \cdot |^{-(|\beta|-l)} \partial^{\alpha-\beta} f \Vert_{L^2(\R^n)} \\
&\leq \frac{C}{R^2} \Vert f \Vert_{\dot{H}^{k-l}(\R^n)}\\
&\leq \frac{C}{R^2} \Vert f \Vert_{X_s^k(\R^n)},
\end{align*}
for some $C >0$ and $l \geq 0$ satisfying $0 \leq |\beta|-l < \frac{n}{2}$ and $k-l \geq s$, which justify the application of Hardy's inequality and the embedding $X_s^k(\R^n) \hookrightarrow \dot{H}_r^{k-l}(\R^n)$.

Finally, by putting everything together, the density of $C^{\infty}_{c,r}(\R^n) \subseteq X_s^k(\R^n)$ and the closedness of $(1-\mathcal{L})^{\frac{j}{2}}$ for $j \in \N_0$, we obtain
\begin{align*}
\Vert Vf \Vert_{X_s^k(\R^n)} \leq C_R \sum_{j=0}^k \Vert f \Vert_{\mathcal{G}((1-\mathcal{L})^{\frac{j}{2}})} + \frac{C}{R^{\varepsilon_V}} \Vert f \Vert_{X_s^k(\R^n)},
\end{align*}
for all $f \in X_s^k(\R^n)$ with suitably chosen $C_R >0$, $C >0$ and some $\varepsilon_V >0$.
\end{proof}

\begin{proposition} \label{Semigroupdecaypropo}
There exists $\omega >0$ such that 
\begin{align} \label{Semigroupdecay}
\Vert S_{X_s^k}(\tau) (1- \mathcal{P}_{X_s^k}) f \Vert_{X_s^k(\R^n)} \lesssim e^{- \omega \tau} \Vert (1- \mathcal{P}_{X_s^k}) f \Vert_{X_s^k(\R^n)} ,
\end{align}
for all $\tau \geq 0$ and $f \in X_s^k(\R^n)$.
\end{proposition}

\begin{proof} 
Let $f \in C^{\infty}_{c,r}(\R^n)$ and $\tau \geq 0$. It is easy to see that $\tilde{f} := (1- \mathcal{P}_{X_s^k}) f$ belongs to $\mathcal{D}(\mathcal{L}_{X_s^k})$. First note that by standard properties of the Fourier transform we have
\begin{align*}
\langle \mathcal{L}_0 f, f \rangle_{\dot{H}^s(\R^n)} &= \langle |\cdot|^s \mathcal{F}( \Delta f - \Lambda f), |\cdot|^s \mathcal{F}f \rangle_{L^2(\R^n)}\\
&=- \Vert f \Vert_{\dot{H}^{s+1}(\R^n)}^2 + \tfrac{1}{2}\left ( \tfrac{n}{2}-1-s \right ) \Vert f \Vert_{\dot{H}^{s}(\R^n)}^2 \leq -\omega_s \Vert f \Vert_{\dot{H}^s(\R^n)}^2,
\end{align*}
for all $f \in C^{\infty}_{c,r}(\R^n)$ with $\omega_s := \frac{1}{2}\bigl( s + 1 -\frac{n}{2} \bigl) >0$ and the same holds for the $\dot{H}^k(\R^n)-$norm with $\omega_k:= \frac{1}{2}\bigl( k + 1 -\frac{n}{2} \bigl) > 0$ instead. By the density of $C^{\infty}_{c,r}(\R^n)$ in $X_s^k(\R^n)$ and the closedness of $\mathcal{L}_0|_{X_s^k}$ we infer
\begin{align*}
\langle \mathcal{L}_0 f, f \rangle_{X_s^k(\R^n)} \leq - \omega_s \Vert f \Vert^2_{X_s^k(\R^n)},
\end{align*} 
for all $f \in X_s^k(\R^n)$, as one has $\omega_s < \omega_k$. This implies
\begin{align*}
\Vert S_{X_s^k}(\tau) \tilde{f} \Vert_{X_s^k(\R^n)} \frac{d}{d \tau} \Vert S_{X_s^k}(&\tau) \tilde{f} \Vert_{X_s^k(\R^n)}\\
&= \frac{1}{2} \frac{d}{d \tau} \Vert S_{X_s^k}(\tau) \tilde{f} \Vert^2_{X_s^k(\R^n)}\\
&= \langle \mathcal{L}_{X_s^k} S_{X_s^k}(\tau) \tilde{f}, S_{X_s^k}(\tau) \tilde{f} \rangle_{X_s^k(\R^n)}\\
&\leq - \omega_s \Vert S_{X_s^k}(\tau) \tilde{f} \Vert_{X_s^k(\R^n)}^2 + |\langle V S_{X_s^k}(\tau) \tilde{f} , S_{X_s^k}(\tau) \tilde{f} \rangle_{X_s^k(\R^n)} |\\
&\leq - \omega_s \Vert S_{X_s^k}(\tau) \tilde{f} \Vert_{X_s^k(\R^n)}^2 + \Vert V S_{X_s^k}(\tau) \tilde{f} \Vert_{X_s^k(\R^n)} \Vert S_{X_s^k}(\tau) \tilde{f} \Vert_{X_s^k(\R^n)},
\end{align*}
for all $f \in C^{\infty}_{c,r}(\R^n)$.
Dividing both sides by the norm of $S_{X_s^k}(\tau) \tilde{f}$ togehter with the results from Lemma \ref{EstimateVf}, Corollary \ref{Estimatesingraph} and Lemma \ref{graphnormestimate} we obtain
\begin{align*}
\frac{d}{d \tau} \Vert S_{X_s^k}(\tau) \tilde{f} \Vert_{X_s^k(\R^n)} &\leq \left( \frac{C}{R^{\varepsilon_V}} - \omega_s \right) \Vert S_{X_s^k}(\tau) \tilde{f} \Vert_{X_s^k(\R^n)} + C_R \sum_{j=0}^k \Vert S_{X_s^k}(\tau) \tilde{f} \Vert_{\mathcal{G}((1-\mathcal{L})^{\frac{j}{2}})}\\
&\leq \left( \frac{C}{R^{\varepsilon_V}} - \omega_s \right) \Vert S_{X_s^k}(\tau) \tilde{f} \Vert_{X_s^k(\R^n)} + C_R e^{-\omega_0 \tau} \Vert \tilde{f} \Vert_{\mathcal{G}((1-\mathcal{L})^{\frac{j}{2}})}\\
&\leq \left( \frac{C}{R^{\varepsilon_V}} - \omega_s \right) \Vert S_{X_s^k}(\tau) \tilde{f} \Vert_{X_s^k(\R^n)} + C_R e^{-\omega_0 \tau} \Vert \tilde{f} \Vert_{X_s^k(\R^n)}.
\end{align*}
Next we choose $R>0$ large enough to get
\begin{align} \label{ChooseR}
\frac{d}{d \tau} \Vert S_{X_s^k}(\tau) \tilde{f} \Vert_{X_s^k(\R^n)} \leq -\frac{\omega_s}{2} \Vert S_{X_s^k}(\tau) \tilde{f} \Vert_{X_s^k(\R^n)} + C e^{-\omega_0 \tau} \Vert \tilde{f} \Vert_{X_s^k(\R^n)},
\end{align}
which is equivalent to
\begin{align*}
\frac{d}{d \tau} \bigl[ e^{\frac{\omega_s}{2} \tau} \Vert S_{X_s^k}(\tau) \tilde{f} \Vert_{X_s^k(\R^n)} \bigl] \leq C e^{(\frac{\omega_s}{2} -\omega_0) \tau} \Vert \tilde{f} \Vert_{X_s^k(\R^n)}.
\end{align*}
Integrating both sides from $0$ to $\tau$ yields
\begin{align*}
e^{\frac{\omega_s}{2} \tau} \Vert S_{X_s^k}(\tau) \tilde{f} \Vert_{X_s^k(\R^n)} - \Vert \tilde{f} \Vert_{X_s^k(\R^n)} \leq C \Vert \tilde{f} \Vert_{X_s^k(\R^n)} \tfrac{e^{(\frac{\omega_s}{2} -\omega_0) \tau}-1}{\frac{\omega_s}{2} -\omega_0},
\end{align*}
which in turn implies
\begin{align*}
\Vert S_{X_s^k}(\tau) \tilde{f} \Vert_{X_s^k(\R^n)} &\leq \Vert \tilde{f} \Vert_{X_s^k(\R^n)} \bigl( e^{-\frac{\omega_s}{2} \tau} + \tfrac{2C}{\omega_s - 2\omega_0} \bigl( e^{-\omega_0 \tau} - e^{-\frac{\omega_s}{2} \tau} \bigl) \bigl) \\
&\lesssim e^{-\omega \tau} \Vert \tilde{f} \Vert_{X_s^k(\R^n)},
\end{align*}
with $\omega := \min \{ \omega_0, \frac{\omega_s}{2} \} >0$ for all $f \in C^{\infty}_{c,r}(\R^n)$. Note that in case $\omega_s = 2 \omega_0$ one has to choose another $R>0$ in \eqref{ChooseR} to avoid this scenario. The statement now follows by density of $C^{\infty}_{c,r}(\R^n)$ in $X_s^k(\R^n)$. 
\end{proof}

\section{The nonlinear time evolution}\label{Sec:Nonlin}

 In the following, we restrict ourselves to real-valued functions from $X_s^k(\R^n)$. We note that by Lemma \ref{InXsk}, the blowup profile satisfies $\phi(|\cdot|) \in X_s^k(\R^n)$, see \eqref{BlowupSol_Trans}.

\subsection{Estimates for the nonlinearity and the initial data operator}

\begin{lemma} \label{Nonlinearityestimate}
The nonlinearity $\mathcal{N}$ from \eqref{Nonlinearity} extends to a map $\mathcal{N}: X_s^k(\R^n) \to  X_s^k(\R^n)$ satisfying 
\begin{align*}
\Vert \mathcal{N}(f) - \mathcal{N}(h) \Vert_{X_s^k(\R^n)} \leq \gamma(\|f\|_{X_s^k(\R^n)},\|g\|_{X_s^k(\R^n)} ) (\Vert f \Vert_{X_s^k(\R^n)} + \Vert h \Vert_{X_s^k(\R^n)}) \Vert f-h \Vert_{X_s^k(\R^n)},
\end{align*}
for all $f,h \in X_s^k(\R^n)$, where $\gamma: [0,\infty) \times [0,\infty) \to [0,\infty)$ is a continuous function.
\end{lemma}

\begin{proof}
By density if suffices to prove  the inequality for functions belonging to $C^{\infty}_{c,r}(\R^n)$.  
We first establish an auxiliary identity for functions $u \in C^3(\R)$ with $u^{\prime \prime}(0) = 0$. Let $a, b, c \in \R$, then by three times application of the Fundamental Theorem of Calculus we have 
\begin{align*}
&u(a+c) - u(a+b) - u^{\prime}(a)(c-b)\\
&=(c-b) \int_{0}^1 (b + x(c-b)) \int_0^1 (a+ w(b+x(c-b))) \int_0^1 u^{\prime \prime \prime} (z (a+ w(b+x(c-b)))) dz dw dx.
\end{align*}
For the nonlinearity we obtain
\begin{align*}
&[\mathcal{N}(f)- \mathcal{N}(h)](y)\\
&= \frac{n-3}{|y|^3}\bigg(  F \Big (|y| \phi(|y|) + |y| f(y) \Big ) + F \Big (|y| \phi(|y|) + |y| h(y) \Big ) - F' \Big (|y| \phi(|y|) \Big )|y| (h(y) - f(y)) \Big ).
\end{align*}
Since $F$ is an odd function we have $F''(0) = 0$ and the above identity yields
\begin{align*}
\frac{[\mathcal{N}(f)- \mathcal{N}(h)](y)}{n-3} \\
= \int_0^1 \int_0^1 \int_0^1 \Big  (h(y)&-f(y)\Big  ) \Big  (f(y)+x (h(y)-f(y))\Big )\Big (\phi(|y|) + w(f(y) + x(h(y)- f(y)))\Big )\\
&\cdot F^{(3)}\bigg (|y| z \Big(\phi(|y|) + w \big (f(y)+x ( h(y)- f(y)) \big ) \Big )\bigg  ) dz dw dx.
\end{align*}
Note that $F^{(3)}$ is an odd function having all derivatives bounded.  We use the generalized Schauder estimate given in \cite{WMglobal}, Proposition A.1, which reads
\begin{align}\label{Eq:Schauder}
\begin{split}
\| u_1 u_2 u_3 F^{(3)}(|\cdot| v) \|_{X_s^k(\R^n)} &  \lesssim \| u_1 u_2 u_3 F^{(3)}(|\cdot| v) \|_{\dot H^{k_1}\cap \dot H^k(\R^n)}  \\
& \lesssim  \prod_{i = 1} ^{3} \|u_i \|_{X_s^k(\R^n)} \sum_{j = 0}^{k} \|v \|_{X_s^k(\R^n)}^{2j},
\end{split}
\end{align}
for $k_1 := \lfloor\frac{n}{2}-2 \rfloor$, where $u_1,u_2,u_3, v \in C^{\infty}_{c,r }(\R^n)$ and $v$ is real-valued. By inspection, the same bound holds for $u_3$ being replaced by $\phi(|\cdot|)$, respectively for $v \in C^{\infty}_r(\R^n) \cap X^{k}_s(\R^n)$. The proof of \eqref{Eq:Schauder} is based on Hardy's inequality and a Sobolev inequality for weighted derivatives of radial functions, see \cite{WMglobal}, Proposition B.1. 
\end{proof}
We need the following characterization of the initial data operator. 

\begin{lemma} \label{LemmaInitialdata}
Let $0 < \delta \leq \frac{1}{2}$. Then the map 
\begin{align*}
T \mapsto \mathcal{U}(\varphi, T) : [1-\delta, 1+ \delta] \rightarrow X_s^k(\R^n),
\end{align*}
as defined in \eqref{InitialU} is continuous for $\varphi \in X_s^k(\R^n)$. Furthermore, we have 
\begin{align} \label{Initialestimate}
\Vert \mathcal{U}(\varphi, T) \Vert_{X_s^k(\R^n)} \lesssim \Vert \varphi \Vert_{X_s^k(\R^n)} + |T-1|,
\end{align}
for all $\varphi \in X_s^k(\R^n)$ and $T \in [\frac{1}{2}, \frac{3}{2}]$.
\end{lemma}

\begin{proof}
To show continuity of the given map, let $\varphi \in X_s^k(\R^n)$, $T_1, T_2 \in [1- \delta, 1+ \delta]$ and write 
\begin{align*}
 \mathcal{U}(\varphi, T_1) -  \mathcal{U}(\varphi, T_2) = (\sqrt{T_1} - \sqrt{T_2})[\phi + \varphi](\sqrt{T_1} \cdot) + \sqrt{T_2} ( [\phi + \varphi](\sqrt{T_1} \cdot) - [\phi + \varphi](\sqrt{T_2} \cdot) ).
\end{align*}
Next let $\varepsilon >0$, then there exists $\chi \in C^{\infty}_{c,r}(\R^n)$, such that $\Vert \chi - [\phi + \varphi] \Vert_{X_s^k(\R^n)} \leq \varepsilon$. By this we can write the second term from above as
\begin{align*}
[\phi + \varphi](\sqrt{T_1} \cdot) - [\phi + \varphi](\sqrt{T_2} \cdot) &= ([\phi + \varphi](\sqrt{T_1} \cdot) - \chi(\sqrt{T_1} \cdot)) + (\chi (\sqrt{T_1} \cdot) - \chi (\sqrt{T_2} \cdot))\\
& +( \chi (\sqrt{T_2} \cdot) - [\phi + \varphi](\sqrt{T_2} \cdot)),
\end{align*}
to infer
\begin{align*}
\lim_{T_2 \rightarrow T_1} \Vert  \mathcal{U}(\varphi, T_1) -  \mathcal{U}(\varphi, T_2) \Vert_{X_s^k(\R^n)} \leq C \varepsilon,
\end{align*}
for some $C>0$. Here we used the fact that $\lim_{T_2 \rightarrow T_1} \Vert \chi (\sqrt{T_1} \cdot) - \chi (\sqrt{T_2} \cdot) \Vert_{X_s^k(\R^n)} =0$, as the function $\chi$ is smooth and compactly supported. As $\varepsilon >0$ was chosen arbitrarily, continuity follows. Next we show \eqref{Initialestimate}. Let $\varphi \in X_s^k(\R^n)$ and $T \in [\frac{1}{2}, \frac{3}{2} ]$, then we have
\begin{align*}
\Vert \mathcal{U}(\varphi, T) \Vert_{X_s^k(\R^n)} &= \Vert \sqrt{T} \varphi(\sqrt{T} \cdot) + \sqrt{T} \phi(\sqrt{T} \cdot) - \phi \Vert_{X_s^k(\R^n)} \\
&\leq  \sqrt{T} \Vert \varphi(\sqrt{T} \cdot)  \Vert_{X_s^k(\R^n)} + \sqrt{T} \Vert  \phi(\sqrt{T} \cdot) - \phi \Vert_{X_s^k(\R^n)} + | \sqrt{T} -1| \Vert \phi \Vert_{X_s^k(\R^n)} \\
&\lesssim \Vert \varphi \Vert_{X_s^k(\R^n)} + |T-1| \Vert \phi \Vert_{X_s^k(\R^n)} + \Vert  \phi(\sqrt{T} \cdot) - \phi \Vert_{X_s^k(\R^n)}.
\end{align*}
To estimate the remaining term we do the following. For $z \in \R^+$ we obtain
\begin{align*}
\phi( \sqrt{T} z) - \phi (z) = z (\sqrt{T}-1) \int_0^1 \phi^{\prime}(z ((\sqrt{T} -1) \tau +1)) d \tau,
\end{align*}
by the Fundamental Theorem of Calculus, which implies
\begin{align*}
\frac{\Vert \phi(\sqrt{T} \cdot) - \phi \Vert_{X_s^k(\R^n)}} {|\sqrt{T}-1|} &= \tilde{C} \bigg\Vert \int_0^1 \frac{|1+(\sqrt{T}-1) \tau| | \cdot |^2}{(|1+(\sqrt{T}-1) \tau|^2 | \cdot |^2 +b)^{\frac{3}{2}}} d \tau \bigg\Vert_{X_s^k(\R^n)}\\
&\leq \tilde{C} \int_0^1 \frac{1}{|t_{\tau}|} \bigg\Vert \frac{|t_{\tau}|^2 | \cdot |^2}{(|t_{\tau}|^2 | \cdot |^2 +b)^{\frac{3}{2}}} \bigg\Vert_{X_s^k(\R^n)} d \tau,
\end{align*}
for some $\tilde{C}>0$ and $t_{\tau} := 1+(\sqrt{T}-1) \tau$. Note that 
\begin{align*}
\bigg\Vert \frac{|t_{\tau}|^2 | \cdot |^2}{(|t_{\tau}|^2 | \cdot |^2 +b)^{\frac{3}{2}}} \bigg\Vert_{X_s^k(\R^n)} \leq \max \{ |t_{\tau}|^{s-\frac{n}{2}}, |t_{\tau}|^{k-\frac{n}{2}} \} \bigg\Vert \frac{| \cdot |^2}{(|\cdot |^2 +b)^{\frac{3}{2}}} \bigg\Vert_{X_s^k(\R^n)} \lesssim 1,
\end{align*} 
by Lemma \ref{InXsk}, as we have $\bigg| \partial^{\alpha} \bigg( \frac{| \cdot |^2}{(|\cdot |^2 +b)^{\frac{3}{2}}} \bigg) (x) \bigg| \lesssim \langle x \rangle^{-1-|\alpha|}$ for $x \in \R^n$ and $\alpha \in \N_0^n$. This shows
\begin{align*}
\Vert \phi(\sqrt{T} \cdot) - \phi \Vert_{X_s^k(\R^n)} \lesssim |\sqrt{T}-1| \lesssim |T-1|,
\end{align*}
which in turn implies
\begin{align*}
\Vert \mathcal{U}(\varphi, T) \Vert_{X_s^k(\R^n)} \lesssim \Vert \varphi \Vert_{X_s^k(\R^n)} + |T-1|,
\end{align*}
as claimed.
\end{proof}

\subsection{Construction of strong solutions}

To show existence of strong solutions to \eqref{Centralproblem} we consider Duhamel's formula 
\begin{align} \label{Duhamel}
\varphi(\tau) = S_{X_s^k}(\tau) \mathcal{U}(\varphi_0, T) + \int_0^{\tau} S_{X_s^k}(\tau - \tau^{\prime}) \mathcal{N}(\varphi(\tau^{\prime})) d \tau^{\prime}, \quad \tau \geq 0.
\end{align}
We introduce the Banach space 
\begin{align*}
\mathcal{X}_s^k := \{ \varphi \in C([0, \infty), X_s^k(\R^n)) : \Vert \varphi \Vert_{\mathcal{X}_s^k} := \sup_{\tau \geq 0} e^{\omega \tau} \Vert \varphi(\tau) \Vert_{X_s^k} < \infty \}, 
\end{align*}
where $\omega >0$ is the constant from Proposition \ref{Semigroupdecay}. We denote the ball in $\mathcal{X}_s^k$ with radius $\delta >0$ by $\mathcal{X}_s^k(\delta) := \{ \varphi \in \mathcal{X}_s^k : \Vert \varphi \Vert_{\mathcal{X}_s^k} \leq \delta \}$. 
To run a fixed point argument, we define the operator 
\begin{align*}
[K(\varphi, u)](\tau) := S_{X_s^k}(\tau)\big (u - \mathcal{C}(\varphi, u)\big) + \int_0^{\tau} S_{X_s^k}(\tau - \tau^{\prime}) \mathcal{N}(\varphi(\tau^{\prime})) d \tau^{\prime},
\end{align*}
for $\varphi \in \mathcal{X}_s^k$, $u \in X_s^k(\R^n)$ and $\tau \geq 0$, where 
\begin{align*}
\mathcal{C}(\varphi, u) := \mathcal{P}_{X_s^k} u + \int_0^{\infty} e^{-\tau^{\prime}} \mathcal{P}_{X_s^k} \mathcal{N}(\varphi(\tau^{\prime})) d \tau^{\prime},
\end{align*}
is a correction term on the unstable subspace $\mathcal{P}_{X_s^k} \bigl( X_s^k(\R^n) \bigl)$. 

\begin{lemma} \label{ContractionofK}
For all sufficiently small $\delta >0$, all sufficiently large $c>0$ and $u \in X_s^k(\R^n)$ with $\Vert u \Vert_{X_s^k(\R^n)} \leq \frac{\delta}{c}$, the operator $K(\cdot , u)$ maps the ball $\mathcal{X}_s^k(\delta)$ into itself and satisfies
\begin{align} \label{contraction}
\Vert K(\varphi_1, u) - K(\varphi_2, u) \Vert_{\mathcal{X}_s^k} \leq \frac{1}{2} \Vert \varphi_1 - \varphi_2 \Vert_{\mathcal{X}_s^k},
\end{align}
for all $\varphi_1, \varphi_2 \in \mathcal{X}_s^k(\delta)$ and all $u \in X_s^k(\R^n)$.
\end{lemma}

\begin{proof}
First note that we have
\begin{align*}
[K(\varphi, u)](\tau) &= (1- \mathcal{P}_{X_s^k}) S_{X_s^k}(\tau)  u + \int_0^{\tau} (1- \mathcal{P}_{X_s^k}) S_{X_s^k}(\tau -\tau^{\prime}) \mathcal{N}(\varphi(\tau^{\prime})) d \tau^{\prime} \\
&-\int_{\tau}^{\infty} e^{\tau-\tau^{\prime}} \mathcal{N}(\varphi(\tau^{\prime})) d \tau^{\prime},
\end{align*}
for $\varphi \in \mathcal{X}_s^k$ and $u \in X_s^k(\R^n)$, which yields the bound
\begin{align*}
\Vert [K(\varphi, u)](\tau) \Vert_{X_s^k(\R^n)} &\lesssim e^{-\omega \tau} \Vert u \Vert_{X_s^k(\R^n)} + \int_0^{\tau} e^{-\omega(\tau - \tau^{\prime})} \Vert \varphi(\tau^{\prime})\Vert_{X_s^k(\R^n)}^2 d \tau^{\prime} \\
&+ \int_{\tau}^{\infty} e^{-(\tau^{\prime} -\tau)} \Vert \varphi(\tau^{\prime})\Vert_{X_s^k(\R^n)}^2 d \tau^{\prime}\\
&\lesssim e^{-\omega \tau} \delta \left ( \tfrac{1}{c} + \tfrac{\delta}{\omega} + \tfrac{\delta}{1+2 \omega} \right),
\end{align*}
for $\varphi \in \mathcal{X}_s^k(\delta)$ and $u \in X_s^k(\R^n)$ satisfying $\Vert u \Vert_{X_s^k(\R^n)} \leq \frac{\delta}{c}$ by Proposition \ref{Semigroupdecaypropo} and Lemma \ref{Nonlinearityestimate}, which yields
\begin{align*}
\Vert K(\varphi, u) \Vert_{\mathcal{X}_s^k} \lesssim \delta
\end{align*}
for suitably chosen $\delta, c >0$.\\
\\
Next let $\varphi_1, \varphi_2 \in \mathcal{X}_s^k(\delta)$, $u \in X_s^k(\R^n)$ and $\tau \geq 0$, then we have
\begin{align*}
\Vert \mathcal{N}(\varphi_1(\tau)) - \mathcal{N}(\varphi_2(\tau)) \Vert_{X_s^k(\R^n)} &\lesssim \delta e^{-\omega \tau} \Vert \varphi_1(\tau) - \varphi_2(\tau) \Vert_{X_s^k(\R^n)}\\
& \lesssim \delta e^{-2\omega \tau} \Vert \varphi_1 - \varphi_2 \Vert_{\mathcal{X}_s^k},
\end{align*}
by Lemma \ref{Nonlinearityestimate}, which implies
\begin{align*}
\Vert K(\varphi_1, u) - K(\varphi_2, u) \Vert_{X_s^k(\R^n)} &\lesssim e^{-\omega \tau} \delta \left ( \tfrac{1}{\omega} + \tfrac{1}{1+2 \omega} \right ) \Vert \varphi_1 - \varphi_2 \Vert_{\mathcal{X}_s^k}.
\end{align*}
By choosing $\delta >0$ sufficiently small we infer
\begin{align*}
\Vert K(\varphi_1, u) - K(\varphi_2, u) \Vert_{\mathcal{X}_s^k} \leq \frac{1}{2} \Vert \varphi_1 - \varphi_2 \Vert_{\mathcal{X}_s^k},
\end{align*} 
and this shows \eqref{contraction}.
\end{proof}

With Lemma \ref{ContractionofK} at hand we are able to construct strong solutions to \eqref{Centralproblem}.

\begin{theorem} \label{Theoremsolution}
There exists $M>0$ sufficiently large and $\delta >0$ sufficiently small, such that for all real-valued $\varphi_0 \in X_s^k(\R^n)$ with $\Vert \varphi_0 \Vert_{X_s^k} \leq \frac{\delta}{M^2}$, there exists a $T = T(\varphi_0) \in [1- \frac{\delta}{M}, 1+ \frac{\delta}{M} ]$ and a unique solution $\varphi \in C([0, \infty), X_s^k(\R^n))$ satisfying \eqref{Duhamel} for all $\tau \geq 0$, such that 
\begin{align} \label{estimateforvarphi}
\Vert \varphi(\tau) \Vert_{X_s^k(\R^n)} \leq \delta e^{- \omega \tau}, \quad \forall \tau \geq 0.
\end{align}
\end{theorem}

\begin{proof}
Let $\varphi_0 \in X_s^k(\R^n)$ be such that $\Vert \varphi_0 \Vert_{X_s^k(\R^n)} \leq \frac{\delta}{M^2}$. Using estimate \eqref{Initialestimate} in Lemma \ref{LemmaInitialdata} we  choose $M>0$ sufficiently large to obtain
\begin{align*}
\Vert \mathcal{U}(\varphi_0, T) \Vert_{X_s^k(\R^n)} \leq \frac{\delta}{c}, 
\end{align*}
for all $T \in [1- \frac{\delta}{M}, 1+ \frac{\delta}{M} ]$ with the constant $c>0$ from Lemma \ref{ContractionofK}. Applying Banach's fixed point theorem we infer that for every $T \in [1- \frac{\delta}{M}, 1+ \frac{\delta}{M} ]$, there exists a unique solution $\varphi_T \in \mathcal{X}_s^k(\delta)$ to the equation
\begin{align*}
\varphi(\tau) = [K(\varphi, \mathcal{U}(\varphi_0, T))](\tau), \quad \tau \geq 0.
\end{align*}
Furthermore, the so defined map $T \mapsto \varphi_T$ is continuous. To see this, note that we have
\begin{align*}
\Vert \varphi_{T_1} - \varphi_{T_2} \Vert_{\mathcal{X}_s^k} &\leq \Vert K(\varphi_{T_1},\mathcal{U}(\varphi_0, T_1)) -  K(\varphi_{T_1},\mathcal{U}(\varphi_0, T_2)) \Vert{\mathcal{X}_s^k}\\
&+ \Vert K(\varphi_{T_1},\mathcal{U}(\varphi_0, T_2)) - K(\varphi_{T_2},\mathcal{U}(\varphi_0, T_2)) \Vert{\mathcal{X}_s^k}\\
&\leq \Vert \mathcal{U}(\varphi_0, T_1) - \mathcal{U}(\varphi_0, T_2) \Vert_{X_s^k(\R^n)} + \frac{1}{2} \Vert \varphi_{T_1} - \varphi_{T_2} \Vert_{\mathcal{X}_s^k},
\end{align*}
which together with Lemma \ref{LemmaInitialdata} implies the claim.\\
\\
To conclude existence of a solution to \eqref{Duhamel}, we need to find a $T = T(\varphi_0)$ such that the correction term vanishes, i.e. $\mathcal{C}(\varphi_{T(\varphi_0)}, \mathcal{U}(\varphi_0, T)) =0$. Note that this is equivalent to
\begin{align} \label{Equitozero}
\left \langle \mathcal{P}_{X_s^k} \mathcal{U}(\varphi_0, T) + \int_0^{\infty} e^{-\tau^{\prime}} \mathcal{N}(\varphi_{T(\varphi_0)}(\tau^{\prime})) d \tau^{\prime}, G \right \rangle_{\mathcal{H}} =0.
\end{align}
A Taylor expansion in $T=1$ reveals that 
\begin{align*}
\mathcal{U}(\varphi_0, T) = \sqrt{T} \varphi_0(\sqrt{T} \cdot) + C (T-1) G + (T-1)^2 R(T, \cdot),
\end{align*}
for some $C \neq 0$, where the remainder is continuous in $T$ and satisfies $\Vert R(T, \cdot) \Vert_{X_s^k(\R^n)} \lesssim 1$ for all $T \in [1- \frac{\delta}{M}, 1+ \frac{\delta}{M}]$. With this at hand we obtain
\begin{align*}
\langle \mathcal{P}_{X_s^k} \mathcal{U}(\varphi_0, T), G \rangle_{\mathcal{H}} = C(T-1) + f(T),
\end{align*}
for a continuous function $f$ satisfying $|f(T)| \lesssim \frac{\delta}{M^2} + \delta^2$ by the condition on the initial data and $T$. Lemma \ref{Nonlinearityestimate} implies that \eqref{Equitozero} is equivalent to 
\begin{align}\label{fixedpointT}
T= F(T) +1
\end{align} for a suitable function $F$, that is continuous and satisfies $|F(T)| \lesssim \frac{\delta}{M^2} + \delta^2$. By choosing $M$ sufficiently large and $\delta$ sufficiently small we infer $|F(T)| \leq \frac{\delta}{M}$, which in turn implies that the map $T \mapsto F(T) +1$ maps the interval $[1- \frac{\delta}{M}, 1+ \frac{\delta}{M}]$ into itself. The Intermediate Value Theorem  now implies the existence of a solution to \eqref{fixedpointT}, which determines $\varphi_T$ solving to \eqref{Duhamel}. By construction, $\Vert \varphi_T(\tau) \Vert_{X_s^k(\R^n)} \leq \delta e^{- \omega \tau}$ for all $\tau \geq 0$.

To show uniqueness, let $\tilde \varphi \in C([0, \infty), X_s^k(\R^n))$ be another solution to \eqref{Duhamel}, $\varphi_T \neq \tilde \varphi $. Then, by the fact that both solutions exhibit the same initial data, there exists $\varepsilon \in (0, \frac{1- \delta}{2})$ and $\tau_0 > 0$ such that
\begin{align} \label{Contradicting}
\Vert \varphi_T(\tau_0) - \tilde \varphi (\tau_0) \Vert_{X_s^k(\R^n)} > \varepsilon,
\end{align}
and 
\begin{align*}
\Vert \varphi_T(\tau) - \tilde \varphi(\tau) \Vert_{X_s^k(\R^n)} <2 \varepsilon \quad \text{for $\tau \in [0, \tau_0]$}.
\end{align*}
We obtain
\begin{align*}
\Vert \tilde \varphi(\tau) \Vert_{X_s^k(\R^n)} \leq \Vert \tilde \varphi(\tau) - \varphi_T(\tau) \Vert_{X_s^k(\R^n)} + \Vert \varphi_T(\tau) \Vert_{X_s^k(\R^n)} < 2 \varepsilon + \delta e^{- \omega \tau} < 1,
\end{align*}
for $\tau \in [0, \tau_0]$. For the semigroup $( S_{X_s^k}(\tau))_{\tau \geq 0}$, we have the bound  $\|S_{X_s^k}(\tau)\| \leq M e^{\tau}$ for $\tau \geq 0$ and some $M > 0$.  Together with Lemma \eqref{Nonlinearityestimate} 
\begin{align*}
\Vert \varphi_T(\tau) - \tilde \varphi(\tau) \Vert_{X_s^k(\R^n)} \lesssim (e^{ \tau} -1) \sup_{\tau^{\prime} \in [0, \tau]} \Vert \varphi_T(\tau^{\prime}) - \tilde \varphi(\tau^{\prime}) \Vert_{X_s^k(\R^n)},
\end{align*}
for $\tau \in [0, \tau_0]$. By a suitable choice of $\tau_1 \in (0, \tau_0]$ we obtain
\begin{align*}
\sup_{\tau \in [0, \tau_1]} \Vert \varphi_T(\tau) - \tilde \varphi(\tau) \Vert_{X_s^k(\R^n)} \leq \frac{1}{2} \sup_{\tau \in [0, \tau_1]} \Vert \varphi_T(\tau) - \tilde \varphi(\tau) \Vert_{X_s^k(\R^n)},
\end{align*}
which shows $\varphi_T = \tilde \varphi$ in $[0, \tau_1]$. Iterating this argument yields $\varphi_T = \tilde \varphi$ in $[0, \tau_0]$, which contradicts \eqref{Contradicting}.
\end{proof}

\begin{remark}
The unique strong solution resulting from Theorem \ref{Theoremsolution} is real-valued, since the subspace of real-valued functions in $X_s^k(\R^n)$ is invariant under the action of $S_{X_s^k}$ and $\mathcal{P}_{X_s^k}$.
\end{remark}

\subsection{Classical solutions}

In this section, we show that initial data belonging to the Schwartz space generates classical solutions to \eqref{Centralproblem}, which are furthermore smooth.

\begin{proposition} \label{Upgrade}
If $\varphi_0$ from Theorem \ref{Theoremsolution} belongs to $\mc S_r(\R^n)$, then the unique strong solution $\varphi$ to \eqref{Duhamel} belongs to $C^{\infty}([0, \infty) \times \R^n)$ and satisfies \eqref{Centralproblem} in the classical sense.
\end{proposition}

\begin{proof}
The regularity of $\mathcal{U}(\varphi_0, T)$ implies $\mathcal{U}(\varphi_0, T) \in \mathcal{D}(\mathcal{L}_{X_s^k})$, which together with the locally Lipschitz continuity of $\mathcal{N}$ yields that the unique strong solution $\varphi$ to \eqref{Duhamel} is indeed a classical solution (see, e.g., \cite{Cazenave}, p.~60, Proposition 4.3.9). This means $\varphi \in C([0, \infty), \mathcal{D}(\mathcal{L}_{X_s^k})) \cap C^1([0, \infty), X_s^k(\R^n))$ and $\varphi$ solves 
\begin{align} \label{solutionequation}
\partial_{\tau} \varphi(\tau) = \mathcal{L} \varphi(\tau) + \mathcal{N}(\varphi(\tau)), \quad \text{for $\tau \geq 0$},
\end{align}
in $X_s^k(\R^n)$. By the continuous embedding $X_s^k(\R^n) \hookrightarrow L^{\infty}(\R^n)$ (see Lemma \ref{embeddinglemma}) we infer that \eqref{solutionequation} holds pointwise. As the operator $\mathcal{L}$ can be decomposed as $\mathcal{L} = \mathcal{L}_0 + L_1$ with $L_1 $ bounded on $X_s^k(\R^n)$, it follows that $\varphi$ satisfies 
\begin{align*}
\varphi(\tau) = S_0(\tau) \mathcal{U}(\varphi_0, T) + \int_0^{\tau} S_0(\tau - \tau^{\prime}) (L_1 \varphi(\tau^{\prime}) + \mathcal{N}(\varphi(\tau^{\prime}))) d \tau^{\prime}, \quad \text{for $\tau \geq 0$}.
\end{align*}
The smoothing properties of the free semigroup, see Appendix \ref{Smoothing_S0}, now imply that  $\varphi(\tau) \in C^{\infty}(\R^n)$ for all $\tau \geq 0$. 

To establish higher regularity in $\tau$ we apply a generalized version of Schwarz's theorem (see, e.g., \cite{Rudin}, p.~235, Theorem 9.41) which allows to interchange the operators $\partial_{\tau}$ and $\mathcal{L}$ applied to $\varphi$. Hence, mixed derivatives of all orders exist, which shows $\varphi \in C^{\infty}([0, \infty) \times \R^n)$.
\end{proof}

\begin{proof}[Proof of Theorem \ref{maintheorem}]
Let $\eta_0 \in \mc S(\R^d)$ with $\eta_0(x) = x \varphi_0$ for a radial Schwartz function $\varphi_0 \in \mc S_r(\R^d)$, $\varphi_0 = \tilde \varphi_0(|\cdot|)$. By
using the representation of the Fourier transform of radial functions, an explicit computation (see \cite{HyperbolicYM}, Proposition A.5), shows that 
\[ \sum_{i=1}^{d} \Big  |[\mc F_{d} \eta_{0,i}](\xi)|^2 \simeq  |\xi|^2 \Big | \big [\mc F_{d+2} \tilde \varphi_0(|\cdot| ) \big ](|\xi|)\Big |^2. \] 
In particular, $\varphi_0$ can also be considered as a radial Schwartz function on $\R^n$ and 
\[ \| \eta_0  \|_{\dot{H}^{s} \cap \dot{H}^{k}(\R^d)}  \simeq \|\tilde \varphi_0(|\cdot| )\|_{\dot{H}^{s} \cap \dot{H}^{k}(\R^{d+2})}. \]
Hence, we can choose $\varepsilon >0$ small enough to guarantee $\Vert \varphi_0 \Vert_{X_s^k(\R^n)} \leq \frac{\delta}{M^2}$, where $\delta, M >0$ are the constants from Theorem \ref{Theoremsolution}. Existence and uniqueness of a strong solution $\varphi \in C([0, \infty), X_s^k(\R^n))$ to \eqref{Duhamel} for some $T \in [1- \frac{\delta}{M}, 1+ \frac{\delta}{M}]$ follows by Theorem \ref{Theoremsolution}.  Proposition \ref{Upgrade} implies $\varphi \in C^{\infty}([0, \infty) \times \R^n)$, solves \eqref{Centralproblem} in the classical sense. Moreover,  $\varphi(\tau,\cdot) = \tilde \varphi(\tau,|\cdot|)$ and by \eqref{estimateforvarphi}
\[  \| \varphi(\tau) \|_{X_s^k(\R^n)} \lesssim \delta e^{- \omega \tau} \] 
for all $\tau \geq 0$. To return to the original variables, note that  
\[\psi(\tau, y): = \phi(|y|) + \varphi(\tau, y)\]
is a classical solution to \eqref{newequation}, that belongs to $C^{\infty}([0, \infty) \times \R^n)$. Furthermore, $\psi$ is radial,  $\psi(\tau,y) = \tilde \psi(\tau,|y|)$, and with 
\[v(t,r) :=  \tfrac{1}{\sqrt{T-t}} \tilde \psi \Big ( \ln \bigl( \tfrac{T}{T-t} \bigl), \tfrac{r}{\sqrt{T-t}} \Big )  =  v_{T}(t,r) + \tilde \varphi \Big (\ln \bigl( \tfrac{T}{T-t} \bigl), \tfrac{r}{\sqrt{T-t}}\Big ) \]
we obtain a solution to \eqref{Eq:CorHMHF} for initial data 
\[ v_0(r) =\phi(r) + \tilde \varphi_0(r). \]
Finally, $U \in C^{\infty}([0,T) \times \R^d)$ defined by 
\begin{align*}
U(t,x) : = x v(t, |x|) 
\end{align*}
solves the original problem \eqref{HMHF} with initial condition $U_0 = \Phi + \eta_0$, where $\Phi(x) = x \phi(|x|)$. Furthermore, we have the decomposition
\begin{align*}
U(t,x) = \tfrac{x}{\sqrt{T-t}} \Big ( \phi \left ( \tfrac{|x|}{\sqrt{T-t}} \right ) + \tilde \varphi  \left ( \ln ( \tfrac{T}{T-t} ), \tfrac{|x|}{\sqrt{T-t}} \right )  \Big ),
\end{align*}
Setting 
\[ \eta \left( t, \tfrac{x}{\sqrt{T-t}} \right) :=  \tfrac{x}{\sqrt{T-t}} \tilde \varphi  \left ( \ln ( \tfrac{T}{T-t} ), \tfrac{|x|}{\sqrt{T-t}} \right )\]
yields \eqref{Main:Sol_Decomp}. Again, by \cite{HyperbolicYM}, Proposition A.5 we infer that 
\begin{align*}
\Vert \eta(t, \cdot) \Vert_{_{\dot{H}^{s} \cap \dot{H}^{k}(\R^d)}} \simeq \left \Vert \varphi \left ( \ln ( \tfrac{T}{T-t} ), \cdot \right ) \right \Vert_{X_s^k(\R^n)} \lesssim \delta e^{-\omega \ln \bigl( \frac{T}{T-t} \bigl)} \lesssim (T-t)^{\omega},
\end{align*}
for $t \in [0,T)$. This implies $\Vert \eta(t, \cdot) \Vert_{\dot{H}^{r}(\R^d)}  \rightarrow 0$ as $t \rightarrow T^{-}$ for $r \in [s,k]$. 
\end{proof}

\appendix
\section{}

\subsection{Proof of Lemma \ref{embeddinglemma}} \label{Embeddingproof}
We first note that in view of the exponential decay of the weight function $\sigma(x) = e^{-\frac{|x|^2}{4}}$ one has the continuous embedding $L^{\infty}(\R^n) \hookrightarrow \mathcal{H}$. So it suffices to show that $X^k_s(\R^n)$ can be continuously embedded into $L^{\infty}(\R^n)$. We have
\begin{align*}
\Vert f \Vert_{L^{\infty}(\R^n)} &\lesssim \Vert \mathcal{F}f \Vert_{L^1(B_1)} + \Vert \mathcal{F}f \Vert_{L^1(B^c_1)} \\
& \lesssim \Vert | \cdot |^{-s} \Vert_{L^2(B_1)} \Vert | \cdot |^s \mathcal{F}f \Vert_{L^2(B_1)} + \Vert | \cdot |^{-k} \Vert_{L^2(B^c_1)} \Vert | \cdot |^k \mathcal{F}f \Vert_{L^2(B^c_1)} \\
& \lesssim \Vert f \Vert_{X_s^k(\R^n)},
\end{align*}
for all $f \in C^{\infty}_{c,r}(\R^n)$. Let $f \in X_s^k(\R^n)$, then there exists a sequence $(f_j)_j \subseteq C^{\infty}_{c,r}(\R^n)$ such that $f_j \rightarrow f$ in $X_s^k(\R^n)$. The above inequality implies that $(f_j)_j$ is a Cauchy sequence in $L^{\infty}(\R^n)$, this means the limit $h := \lim_{j \rightarrow \infty} f_j$ in $L^{\infty}(\R^n)$ exists. This defines a map $i: X_s^k(\R^n) \rightarrow L^{\infty}(\R^n)$, $i(f)=h$. To have an embedding $i$ must be injective.\\
Let $f \in X_s^k(\R^n)$ be such that $i(f)=0$. Then there exists a sequence $(f_j)_j \subseteq C^{\infty}_{c,r}(\R^n)$ with $f_j \rightarrow f$ in $X_s^k(\R^n)$ and $f_j \rightarrow 0$ in $L^{\infty}(\R^n)$. For a Schwartz function $\varphi \in \mc S(\R^n)$ we have
\begin{align*}
| \langle | \cdot |^s \mathcal{F}(f_j), \varphi \rangle_{L^2(\R^n)} | &= | \langle f_j, \mathcal{F}^{-1} (| \cdot |^s \varphi ) \rangle_{L^2(\R^n)} |\\
&\leq \Vert f_j \Vert_{L^{\infty}(\R^n)} \Vert \mathcal{F}^{-1} (| \cdot |^s \varphi ) \Vert_{L^1(\R^n)} \ \rightarrow 0 \quad \text{as $j \rightarrow \infty$}.
\end{align*}
Together with the density of $\mc S(\R^n)$ in $L^2(\R^n)$ we infer $| \cdot |^s \mathcal{F}(f_j) \rightharpoonup 0$ in $L^2(\R^n)$ as $j \rightarrow \infty$. The same holds for the sequence $(| \cdot |^k \mathcal{F}(f_j))_j$. But the strong limit $f_j \rightarrow f$ in $X_s^k(\R^n)$ implies in particular the weak limits $| \cdot |^s \mathcal{F}(f_j) \rightharpoonup f$ and $| \cdot |^k \mathcal{F}(f_j) \rightharpoonup f$ in $L^2(\R^n)$, which is a contradiction to the uniqueness of weak limits.

\subsection{Alternative proof of the algebra property in $X_s^k(\R^n)$}
Let  $s,k \geq 0$ satisfy the conditions  \eqref{sk-range}. Since $X_s^k(\R^n)$ is a space of radial functions, we can give an elementary proof that $X_s^k(\R^n)$ is closed under multiplication, i.e., that we have
\begin{align*}
\Vert f h \Vert_{X_s^k(\R^n)} \lesssim \Vert f \Vert_{X_s^k(\R^n)} \Vert h \Vert_{X_s^k(\R^n)},
\end{align*}
for all $f,h \in X_s^k(\R^n)$. We first observe
\begin{align*}
\Vert f h \Vert_{X_s^k(\R^n)} \lesssim \Vert f h \Vert_{\dot{H}^{\lfloor s \rfloor} \cap \dot{H}^{k}(\R^n)},
\end{align*}
for all $f,h \in C^{\infty}_{c,r}(\R^n)$, where $\lfloor s \rfloor = \lfloor \frac{n}{2} \rfloor -1$. Next we use the following equivalence of norms for integer exponents $l \in \N_0$ 
\begin{align*}
\Vert f h \Vert_{\dot{H}^{l}} \simeq \sum_{\substack{\alpha \in \N_0^n \\ |\alpha| = l}} \Vert \partial^{\alpha} (fh) \Vert_{L^2(\R^n)} &\leq \sum_{\substack{\alpha \in \N_0^n \\ |\alpha| = l}} \sum_{\substack{\beta \in \N_0^n \\ \beta \leq \alpha}} \binom{\alpha}{\beta} \Vert \partial^{\beta} f \partial^{\alpha-\beta} h \Vert_{L^2(\R^n)},
\end{align*}
where $f, h \in C^{\infty}_{c,r}(\R^n)$. Hence, we estimate
\begin{align} \label{Inequalityfh}
\Vert \partial^{\beta} f \partial^{\alpha-\beta} h \Vert_{L^2(\R^n)} \leq \Vert | \cdot |^{-m} \partial^{\beta} f \Vert_{L^2(\R^n)} \Vert | \cdot |^{\frac{n}{2}-(\frac{n}{2} -m)} \partial^{\alpha-\beta} h \Vert_{L^{\infty}(\R^n)},
\end{align}
for $m \geq 0$ by H\"older's inequality. For $l=\lfloor s \rfloor$, we choose $m=s-|\beta|$. To control weighted derivatives in $L^{\infty}(\R^n)$, we use  \cite{WMglobal}, Proposition B.1 and apply Hardy's inequality to get
\begin{align*}
\Vert \partial^{\beta} f \partial^{\alpha-\beta} h \Vert_{L^2(\R^n)} &\leq \Vert | \cdot |^{-s + |\beta|} \partial^{\beta} f \Vert_{L^2(\R^n)} \Vert | \cdot |^{\frac{n}{2} - (\frac{n}{2}- s + |\beta|)} \partial^{\alpha-\beta} h \Vert_{L^{\infty}(\R^n)}\\
&\lesssim \Vert f \Vert_{\dot{H}^s(\R^n)} \Vert h \Vert_{\dot{H}^{\lfloor s \rfloor +\frac{n}{2}-s}(\R^n)}\\
&\lesssim \Vert f \Vert_{X_s^k(\R^n)} \Vert h \Vert_{X_s^k(\R^n)},
\end{align*}
as $\frac{1}{2} < \frac{n}{2} - s + |\beta| < \frac{n}{2}$ and $s \leq \lfloor s \rfloor + \frac{n}{2}-s \leq k$. \\
For $l=k$ one needs to handle the terms $(\partial^{\alpha} f) h$ and $f( \partial^{\alpha} h)$ appearing in the Leibniz rule separately. We have
\begin{align*}
\Vert (\partial^{\alpha} f) h \Vert_{L^2(\R^n)} \leq \Vert \partial^{\alpha} f \Vert_{L^2(\R^n)} \Vert h \Vert_{L^{\infty}(\R^n)} \lesssim \Vert f \Vert_{X_s^k(\R^n)} \Vert h \Vert_{X_s^k(\R^n)},
\end{align*}
by the continuous embedding $X_s^k(\R^n) \hookrightarrow L^{\infty}(\R^n)$ (see Lemma \ref{embeddinglemma}). The same holds for $f( \partial^{\alpha} h)$. Next, if $1 \leq |\beta| \leq k-1$ we can restrict ourself to $1 \leq |\beta| \leq \frac{k-1}{2}$ by symmetry and use inequality \eqref{Inequalityfh} with $m=\frac{n-2}{2}$ to get
\begin{align*}
\Vert \partial^{\beta} f \partial^{\alpha-\beta} h \Vert_{L^2(\R^n)}
&\lesssim \Vert f \Vert_{\dot{H}^{|\beta|+\frac{n-2}{2}}(\R^n)} \Vert h \Vert_{\dot{H}^{k - |\beta| +1}(\R^n)}\\
&\lesssim \Vert f \Vert_{X_s^k(\R^n)} \Vert h \Vert_{X_s^k(\R^n)},
\end{align*}
by Hardy's inequality, Proposition B.1 in \cite{WMglobal}, and interpolation.

\subsection{Proof of Lemma \ref{InXsk}} \label{Prooflemma42}
Let $f \in C^{\infty}_r(\R^n)$ such that $|\partial^{\alpha} f(x)| \lesssim \langle x \rangle^{-1-|\alpha|}$ holds for $\alpha \in \N_0^n$, $|\alpha| \leq k$ and all $x \in \R^n$. Let $\chi \in C^{\infty}_{c,r}(\R^n)$, $\chi \geq 0$ be a cutoff function satisfying 
\begin{equation}
\left\{
\begin{aligned} \chi(x) &= 1, \quad |x| \leq 1,\\
\chi(x) &= 0, \quad |x| \geq 2.
\end{aligned}
\right.
\end{equation}
We define the sequence $(f_j)_j \subset C^{\infty}_{c,r}(\R^n)$ by $f_j(x) := f(x) \chi_j(x)$ for $x \in \R^n$ and $j \in \N$, where $\chi_j(x) := \chi \bigl( \frac{x}{j} \bigl)$. Obviously, this sequence of functions converges to $f$ in $L^{\infty}(\R^n)$, as we have
\begin{align*}
\Vert f - f_j \Vert_{L^{\infty}(\R^n)} = \Vert f (1- \chi \bigl( \tfrac{\cdot}{j} \bigl) )\Vert_{L^{\infty}(\R^n \setminus B_j^n)} \leq \Vert f \Vert_{L^{\infty}(\R^n \setminus B_j^n)} \rightarrow 0, \quad \text{as $j \rightarrow \infty$}.
\end{align*}
Now we use the fact that 
\[ \| f \|_{\dot H^s(\R^n)} \lesssim  \| f \|_{\dot{W}^{\lceil s \rceil, p}(\R^n)} \simeq \sum_{\substack{\alpha \in \N_0^n \\ |\alpha| = \lceil s \rceil}} \Vert \partial^{\alpha} f \Vert_{L^p(\R^n)}, \]
for $\frac{1}{p} = \frac{1}{2} + \frac{\lceil s \rceil - s}{n} $, 
see also \eqref{WspEmb}. By the decay assumption on $f$ one finds that the sequence $(f_j)_j$ is a  Cauchy sequence in $\dot{W}^{\lceil s \rceil, p} (\R^n)$, if and only if $\frac{n}{2}-1 <s$, which is exactly our condition on $s$. In the same way we infer that $(f_j)_j$ is a Cauchy sequence of radial functions in $\dot{H}^k(\R^n)$ and thus also in $X_s^k(\R^n)$. As $X_s^k(\R^n)$ is complete, there exists $\tilde{f} \in X_s^k(\R^n)$ with $\lim_{j \rightarrow \infty} f_j = \tilde{f}$ in $X_s^k(\R^n)$. By the continuous embedding $X_s^k(\R^n) \hookrightarrow L^{\infty}(\R^n)$, see Lemma \ref{embeddinglemma}, we conclude $\tilde{f} =f$, which shows $f \in X_s^k(\R^n)$.

\subsection{Smoothing property of the free semigroup}\label{Smoothing_S0}

Recall that $\varphi$ from Proposition \ref{Upgrade} satisfies 
\begin{align} \label{DuhamelS0}
\varphi(\tau) = S_0(\tau) \mathcal{U}(\varphi_0, T) + \int_0^{\tau} S_0(\tau - \tau^{\prime}) (L_1 \varphi(\tau^{\prime}) + \mathcal{N}(\varphi(\tau^{\prime}))) d \tau^{\prime}, \quad \text{for $\tau \geq 0$}.
\end{align}
We show that $\varphi(\tau)$ belongs to all $X_s^l(\R^n)$ for $l \geq k$ and use the continuous embedding $X_s^l(\R^n) \hookrightarrow C_r^m(\R^n)$ with $l \geq \frac{n}{2}+m$. Let $f \in C^{\infty}_{c,r}(\R^n)$. Using the explicit form $S_0(\tau)$, see  \eqref{S0}, we get 
\begin{align*}
\Vert S_0(\tau) f \Vert_{\dot{H}^{k+1}(\R^n)} &= e^{-\frac{\tau}{2}} \Vert | \cdot |^{k+1} \mathcal{F} \bigl(( H_{\kappa(\tau)} \ast f ) (e^{-\frac{\tau}{2}} \cdot) \bigl) \Vert_{L^2(\R^n)}\\
&= e^{\frac{1}{2} (\frac{n}{2}-1-(k+1))\tau} \Vert | \cdot |^{k+1} \mathcal{F} \bigl( H_{\kappa(\tau)} \ast f \bigl) \Vert_{L^2(\R^n)}\\
&\lesssim e^{\frac{1}{2} (\frac{n}{2}-1-(k+1))\tau} \Vert | \cdot |^{k} \mathcal{F}(f) \Vert_{L^2(\R^n)} \Vert | \cdot | \mathcal{F}(H_{\kappa(\tau)}) \Vert_{L^{\infty}(\R^n)}\\
&\lesssim e^{\frac{1}{2} (\frac{n}{2}-1-(k+1))\tau} \kappa(\tau)^{-\frac{1}{2}} \Vert | \cdot |^{k} \mathcal{F}(f) \Vert_{L^2(\R^n)},
\end{align*}
for $\tau > 0$. Together with \eqref{MapPropS0}, we infer that 
\begin{align*}
\Vert S_0(\tau) f \Vert_{X_s^{k+1}(\R^n)} \lesssim e^{\frac{1}{2} (\frac{n}{2}-1-s)\tau} \kappa(\tau)^{-\frac{1}{2}} \Vert f \Vert_{X_s^k(\R^n)},
\end{align*}
holds for all $f \in X_s^k(\R^n)$ by density.
By \eqref{DuhamelS0} we infer 
\begin{align*}
\Vert \varphi(\tau) \Vert_{X_s^{k+1}(\R^n)} &\leq \Vert S_0(\tau) \mathcal{U}(\varphi_0, T) \Vert_{X_s^{k+1}(\R^n)} + \int_0^{\tau} \Vert S_0(\tau - \tau^{\prime}) (L_1 \varphi(\tau^{\prime}) + \mathcal{N}(\varphi(\tau^{\prime}))) \Vert_{X_s^{k+1}(\R^n)} d \tau^{\prime}\\
&\lesssim e^{\frac{1}{2} (\frac{n}{2}-1-s)\tau} \kappa(\tau)^{-\frac{1}{2}} \Vert \mathcal{U}(\varphi_0, T) \Vert_{X_s^{k}(\R^n)}\\
&+ \int_0^{\tau} e^{\frac{1}{2} (\frac{n}{2}-1-s)(\tau - \tau^{\prime})} \kappa(\tau - \tau^{\prime})^{-\frac{1}{2}} \Vert L_1 \varphi(\tau^{\prime}) + \mathcal{N}(\varphi(\tau^{\prime})) \Vert_{X_s^k(\R^n)} d \tau^{\prime}.
\end{align*}
As $\varphi$ belongs to $C([0,\infty), X_s^k(\R^n))$ the continuity of the operators $L_1$ and $\mathcal{N}$ imply the estimate
\begin{align*}
&\int_0^{\tau} e^{\frac{1}{2} (\frac{n}{2}-1-s)(\tau - \tau^{\prime})} \kappa(\tau - \tau^{\prime})^{-\frac{1}{2}} \Vert L_1 \varphi(\tau^{\prime}) + \mathcal{N}(\varphi(\tau^{\prime})) \Vert_{X_s^k(\R^n)} d \tau^{\prime}\\
&\leq \sup_{\tilde{\tau} \in [0, \tau]} \Vert L_1 \varphi(\tilde{\tau}) + \mathcal{N}(\varphi(\tilde{\tau})) \Vert_{X_s^k(\R^n)} \int_0^{\tau} e^{\frac{1}{2} (\frac{n}{2}-1-s)(\tau - \tau^{\prime})} \kappa(\tau - \tau^{\prime})^{-\frac{1}{2}} d \tau^{\prime}.
\end{align*}
Hence, $\varphi(\tau) \in X_s^{k+1}(\R^n)$ for $\tau \geq 0$ if 
\begin{align} \label{Integral1}
\int_0^{\tau} e^{\frac{1}{2} (\frac{n}{2}-1-s)(\tau - \tau^{\prime})} \kappa(\tau - \tau^{\prime})^{-\frac{1}{2}} d \tau^{\prime} < \infty.
\end{align}
A straightforward calculation shows that the integral in \eqref{Integral1} is equal to $B(\frac{1}{2}, \frac{s}{2}-\frac{n}{4}+\frac{1}{2})$, where $B$ is the standard beta function and its value is well-defined for $s > \frac{n}{2}-1$.
We inductively infer $\varphi(\tau) \in X_s^l(\R^n)$ for all $\tau \geq 0$ and $l \geq k$, which implies the claim.\\

\emph{Acknowledgment.} The first author would like to thank David Wallauch for fruitful discussions.

\pagestyle{plain}
	\bibliography{Literatur}

\begin{thebibliography}{10}

\bibitem{Bie15}
Pawe{\l} Biernat.
\newblock Non-self-similar blow-up in the heat flow for harmonic maps in higher
  dimensions.
\newblock {\em Nonlinearity}, 28(1):167--185, 2015.

\bibitem{BieBiz11}
Pawe{\l} Biernat and Piotr Bizo\'{n}.
\newblock Shrinkers, expanders, and the unique continuation beyond generic
  blowup in the heat flow for harmonic maps between spheres.
\newblock {\em Nonlinearity}, 24(8):2211--2228, 2011.

\bibitem{BieDon18}
Pawe{\l} Biernat and Roland Donninger.
\newblock Construction of a spectrally stable self-similar blowup solution to
  the supercritical corotational harmonic map heat flow.
\newblock {\em Nonlinearity}, 31(8):3543--3566, 2018.

\bibitem{BieDonSch17}
Pawe{\l} Biernat, Roland Donninger, and Birgit Sch\"{o}rkhuber.
\newblock Stable self-similar blowup in the supercritical heat flow of harmonic
  maps.
\newblock {\em Calc. Var. Partial Differential Equations}, 56(6):Art. 171, 31,
  2017.

\bibitem{BieSek19}
Pawe{\l} Biernat and Yukihiro Seki.
\newblock Type {II} blow-up mechanism for supercritical harmonic map heat flow.
\newblock {\em Int. Math. Res. Not. IMRN}, (2):407--456, 2019.

\bibitem{BieSek20}
Pawe{\l} Biernat and Yukihiro Seki.
\newblock Transition of blow-up mechanisms in {$k$}-equivariant harmonic map
  heat flow.
\newblock {\em Nonlinearity}, 33(6):2756--2796, 2020.

\bibitem{BizWas15}
Piotr Bizo\'{n} and Arthur Wasserman.
\newblock Nonexistence of shrinkers for the harmonic map flow in higher
  dimensions.
\newblock {\em Int. Math. Res. Not. IMRN}, (17):7757--7762, 2015.

\bibitem{Cazenave}
Thierry Cazenave and Alain Haraux.
\newblock {\em An introduction to semilinear evolution equations}, volume~13 of
  {\em Oxford Lecture Series in Mathematics and its Applications}.
\newblock The Clarendon Press, Oxford University Press, New York, 1998.
\newblock Translated from the 1990 French original by Yvan Martel and revised
  by the authors.

\bibitem{ChaDinYe92}
Kung-Ching Chang, Wei~Yue Ding, and Rugang Ye.
\newblock Finite-time blow-up of the heat flow of harmonic maps from surfaces.
\newblock {\em J. Differential Geom.}, 36(2):507--515, 1992.

\bibitem{Che17}
Bang-Yen Chen.
\newblock {\em Differential geometry of warped product manifolds and
  submanifolds}.
\newblock World Scientific Publishing Co. Pte. Ltd., Hackensack, NJ, 2017.
\newblock With a foreword by Leopold Verstraelen.

\bibitem{CorGhi89}
Jean-Michel Coron and Jean-Michel Ghidaglia.
\newblock Explosion en temps fini pour le flot des applications harmoniques.
\newblock {\em C. R. Acad. Sci. Paris S\'{e}r. I Math.}, 308(12):339--344,
  1989.

\bibitem{DavDelPes19}
Juan D\'{a}vila, Manuel Del~Pino, Catalina Pesce, and Juncheng Wei.
\newblock Blow-up for the 3-dimensional axially symmetric harmonic map flow
  into {$S^2$}.
\newblock {\em Discrete Contin. Dyn. Syst.}, 39(12):6913--6943, 2019.

\bibitem{DavPinWei20}
Juan D\'{a}vila, Manuel del Pino, and Juncheng Wei.
\newblock Singularity formation for the two-dimensional harmonic map flow into
  {$S^2$}.
\newblock {\em Invent. Math.}, 219(2):345--466, 2020.

\bibitem{DonGlo19}
Roland Donninger and Irfan Glogi\'{c}.
\newblock On the existence and stability of blowup for wave maps into a
  negatively curved target.
\newblock {\em Anal. PDE}, 12(2):389--416, 2019.

\bibitem{EelLem78}
James Eells and Luc Lemaire.
\newblock A report on harmonic maps.
\newblock {\em Bull. London Math. Soc.}, 10(1):1--68, 1978.

\bibitem{EelLem88}
James Eells and Luc Lemaire.
\newblock Another report on harmonic maps.
\newblock {\em Bull. London Math. Soc.}, 20(5):385--524, 1988.

\bibitem{EelSam64}
James Eells, Jr. and Joseph~H. Sampson.
\newblock Harmonic mappings of {R}iemannian manifolds.
\newblock {\em Amer. J. Math.}, 86:109--160, 1964.

\bibitem{Engel}
Klaus-Jochen Engel and Rainer Nagel.
\newblock {\em One-parameter semigroups for linear evolution equations}, volume
  194 of {\em Graduate Texts in Mathematics}.
\newblock Springer-Verlag, New York, 2000.
\newblock With contributions by S. Brendle, M. Campiti, T. Hahn, G. Metafune,
  G. Nickel, D. Pallara, C. Perazzoli, A. Rhandi, S. Romanelli and R.
  Schnaubelt.

\bibitem{Fan99}
Huijun Fan.
\newblock Existence of the self-similar solutions in the heat flow of harmonic
  maps.
\newblock {\em Sci. China Ser. A}, 42(2):113--132, 1999.

\bibitem{Gas02}
Andreas Gastel.
\newblock Singularities of first kind in the harmonic map and {Y}ang-{M}ills
  heat flows.
\newblock {\em Math. Z.}, 242(1):47--62, 2002.

\bibitem{GerGhoMiu17}
Pierre Germain, Tej-Eddine Ghoul, and Hideyuki Miura.
\newblock On uniqueness for the harmonic map heat flow in supercritical
  dimensions.
\newblock {\em Comm. Pure Appl. Math.}, 70(12):2247--2299, 2017.

\bibitem{GhoIbrNgu19}
Tej-eddine Ghoul, Slim Ibrahim, and Van~Tien Nguyen.
\newblock On the stability of type ii blowup for the 1-corotational
  energy-supercritical harmonic heat flow.
\newblock {\em Anal. PDE}, 12(1):113--187, 2019.

\bibitem{WMglobal}
Irfan {Glogi{\'c}}.
\newblock {Globally stable blowup profile for supercritical wave maps in all
  dimensions}.
\newblock {\em arXiv e-prints}, page arXiv:2207.06952, July 2022.

\bibitem{HyperbolicYM}
Irfan Glogi\'{c}.
\newblock Stable blowup for the supercritical hyperbolic {Y}ang-{M}ills
  equations.
\newblock {\em Adv. Math.}, 408(part B):Paper No. 108633, 52, 2022.

\bibitem{YangMills}
Irfan Glogi\'{c} and Birgit Sch\"{o}rkhuber.
\newblock Nonlinear stability of homothetically shrinking {Y}ang-{M}ills
  solitons in the equivariant case.
\newblock {\em Comm. Partial Differential Equations}, 45(8):887--912, 2020.

\bibitem{KellerSegel}
Irfan {Glogi{\'c}} and Birgit {Sch{\"o}rkhuber}.
\newblock {Stable singularity formation for the Keller-Segel system in three
  dimensions}.
\newblock {\em to appear in Arch. Ration. Mech. Anal.}, page arXiv:2209.11206,
  September 2022.

\bibitem{Grafakos}
Loukas Grafakos and Seungly Oh.
\newblock The {K}ato-{P}once inequality.
\newblock {\em Comm. Partial Differential Equations}, 39(6):1128--1157, 2014.

\bibitem{GraHam96}
Matthew Grayson and Richard~S. Hamilton.
\newblock The formation of singularities in the harmonic map heat flow.
\newblock {\em Comm. Anal. Geom.}, 4(4):525--546, 1996.

\bibitem{Har67}
Philip Hartman.
\newblock On homotopic harmonic maps.
\newblock {\em Canadian J. Math.}, 19:673--687, 1967.

\bibitem{HelWoo08}
Fr\'{e}d\'{e}ric H\'{e}lein and John~C. Wood.
\newblock Harmonic maps.
\newblock In {\em Handbook of global analysis}, pages 417--491, 1213. Elsevier
  Sci. B. V., Amsterdam, 2008.

\bibitem{Jos81}
J\"{u}rgen Jost.
\newblock Ein {E}xistenzbeweis f\"{u}r harmonische {A}bbildungen, die ein
  {D}irichletproblem l\"{o}sen, mittels der {M}ethode des {W}\"{a}rmeflusses.
\newblock {\em Manuscripta Math.}, 34(1):17--25, 1981.

\bibitem{LinWan08}
Fanghua Lin and Changyou Wang.
\newblock {\em The analysis of harmonic maps and their heat flows}.
\newblock World Scientific Publishing Co. Pte. Ltd., Hackensack, NJ, 2008.

\bibitem{RapSch13}
Pierre Rapha{\"e}l and Remi Schweyer.
\newblock Stable blowup dynamics for the 1-corotational energy critical
  harmonic heat flow.
\newblock {\em Comm. Pure Appl. Math.}, 66(3):414--480, 2013.

\bibitem{RapSch14}
Pierre Rapha{\"e}l and Remi Schweyer.
\newblock Quantized slow blow-up dynamics for the corotational energy-critical
  harmonic heat flow.
\newblock {\em Anal. PDE}, 7(8):1713--1805, 2014.

\bibitem{Rudin}
Walter Rudin.
\newblock {\em Principles of mathematical analysis}.
\newblock International Series in Pure and Applied Mathematics. McGraw-Hill
  Book Co., New York-Auckland-D\"{u}sseldorf, third edition, 1976.

\bibitem{ShaTah94}
Jalal Shatah and A.~Shadi Tahvildar-Zadeh.
\newblock On the {C}auchy problem for equivariant wave maps.
\newblock {\em Comm. Pure Appl. Math.}, 47(5):719--754, 1994.

\bibitem{Str88}
Michael Struwe.
\newblock On the evolution of harmonic maps in higher dimensions.
\newblock {\em J. Differential Geom.}, 28(3):485--502, 1988.

\bibitem{Str92}
Michael Struwe.
\newblock The evolution of harmonic maps: existence, partial regularity, and
  singularities.
\newblock In {\em Nonlinear diffusion equations and their equilibrium states, 3
  ({G}regynog, 1989)}, volume~7 of {\em Progr. Nonlinear Differential Equations
  Appl.}, pages 485--491. Birkh\"{a}user Boston, Boston, MA, 1992.

\bibitem{Tac85}
Atsushi Tachikawa.
\newblock Rotationally symmetric harmonic maps from a ball into a warped
  product manifold.
\newblock {\em Manuscripta Math.}, 53(3):235--254, 1985.

\bibitem{Tao}
Terence Tao.
\newblock {\em Nonlinear dispersive equations}, volume 106 of {\em CBMS
  Regional Conference Series in Mathematics}.
\newblock Published for the Conference Board of the Mathematical Sciences,
  Washington, DC; by the American Mathematical Society, Providence, RI, 2006.
\newblock Local and global analysis.

\bibitem{Teschl}
Gerald Teschl.
\newblock {\em Mathematical methods in quantum mechanics}, volume~99 of {\em
  Graduate Studies in Mathematics}.
\newblock American Mathematical Society, Providence, RI, 2009.
\newblock With applications to Schr\"{o}dinger operators.

\end{thebibliography}
	\bibliographystyle{plain}

\end{document}